\documentclass{article}
\usepackage{amsthm}
\usepackage{amsbsy}
\usepackage{amssymb}
\usepackage{amsfonts}
\usepackage{amsmath}
\usepackage{mathrsfs}
\usepackage{amscd}
\usepackage[dvips]{lscape}
\usepackage{verbatim} 
\usepackage{enumerate}

\theoremstyle{definition}
\newtheorem{thm}{Theorem}[section]
\newtheorem{lem}[thm]{Lemma}
\newtheorem{prp}[thm]{Proposition}
\newtheorem{dfn}[thm]{Definition}
\newtheorem{cor}[thm]{Corollary}

\newtheorem{rmk}[thm]{Remark}
\newtheorem{ntn}[thm]{Notation}
\newtheorem{exa}[thm]{Example}

\newcommand{\beq}{\begin{equation}}
\newcommand{\eeq}{\end{equation}}
\newcommand{\beqr}{\begin{eqnarray*}}
\newcommand{\eeqr}{\end{eqnarray*}}
\newcommand{\bal}{\begin{align*}} 
\newcommand{\eal}{\end{align*}}
\newcommand{\bei}{\begin{itemize}}
\newcommand{\eei}{\end{itemize}}

\newcommand{\pa}{\precapprox_s}

\newcommand{\af}{\alpha}
\newcommand{\bt}{\beta}

\newcommand{\dt}{\delta}
\newcommand{\ep}{\varepsilon}
\newcommand{\zt}{\zeta}

\newcommand{\io}{\iota}

\newcommand{\sm}{\sigma}

\newcommand{\ta}{\tau}

\newcommand{\Dt}{\Delta}

\newcommand{\Z}{{\mathbb{Z}}}

\newcommand{\C}{{\mathbb{C}}}
\newcommand{\N}{{\mathbb{Z}}_{> 0}}

\newcommand{\Her}{\mathrm{Her}}

\newcommand{\Tr}{{\mathrm{Tr}}}

\newcommand{\dist}{{\mathrm{dist}}}

\newcommand{\diag}{{\mathrm{diag}}}

\newcommand{\card}{{\mathrm{card}}}
\newcommand{\Aut}{{\mathrm{Aut}}}
\newcommand{\Ad}{{\mathrm{Ad}}}

\newcommand{\B}{{\mathrm{B}}}
\newcommand{\Ker}{\mathrm{Ker}}
\newcommand{\E}{\mathrm{E}}
\newcommand{\TA}{\mathrm{TA}}

\newcommand{\QED}{\rule{0.4em}{2ex}}

\newcommand{\wolog}{without loss of generality}
\newcommand{\Wolog}{Without loss of generality}

\newcommand{\ifo}{if and only if}

\newcommand{\ca}{C*-algebra}

\newcommand{\pj}{projection}
\newcommand{\mops}{mutually orthogonal \pj s}

\newcommand{\mvnt}{Murray-von Neumann equivalent}

\begin{document}
\title{Tracial Rokhlin property and non-commutative dimension}
\author{Qingyun Wang}
\maketitle 
\begin{abstract}
Tracial Rokhlin property was introduced by Phillips in \cite{Ph-C} to prove various structure theorems for crossed product. But it is defined for actions on simple C*-algebras only. This paper consists of two major parts: In section 2 and 3, we study the permanence properties and give a complete classification of tracial Rokhlin property for product-type actions; In section 4 and 5, we introduce the weak tracial Rokhlin property for actions on non-simple C*-algebras. We prove that when the action has the weak tracial Rokhlin property and the crossed product is simple, the properties on $A$ of having tracial rank $\leq k$, or real rank 0, or stable rank 1, can be inherited by the crossed product.
\end{abstract}
\section{Introduction}
The Rokhlin property for finite group actions on C*-algebras has been extensively studied since Izumi's paper \cite{Iz1}. It is defined as follows:
\begin{dfn}\label{SRP}
Let $A$ be an infinite dimensional unital \ca, let $\af\colon G \rightarrow \Aut (A)$ be an action of a finite group $G$ on $A$. We say that $\af$  has the {\emph{Rokhlin property}} if for every finite set $F\subset A,$ every $\ep>0$, there are \mops\ $e_g\in A$ for $g\in G$ with $\sum_{g\in G}e_g=1$, such that:
\begin{enumerate}[(1)]
\item
$\|\af_g(e_h)-e_{gh}\|<\ep
$ for all $g, h\in G$.
\item
$\|e_g a-ae_g\|<\ep$ for all $g\in G$ and all $a\in F.$
\end{enumerate}
\end{dfn}
A number of authors have shown that the crossed products by actions with the Rokhlin property preserve many important classes of \ca, such as AF algebras, AI algebras, AT algebras, simple AH algebras with slow dimension growth and real rank 0, $\mathcal{D}$-absorbing C*-algebras for a strongly self-absorbing \ca\ $\mathcal{D}$ and so on. \newline
But the Rokhlin property is very rare.  In the paper \cite{Ph-F}, Phillips  pointed out many obstructions for the Rokhlin property and introduced the tracial Rokhlin property:
\begin{dfn}\label{TRP}
Let $A$ be an infinite dimensional simple unital C*-algebra, let $\af\colon G \rightarrow \Aut(A)$ be an action of a finite group $G$ on $A$. We say that $\af$  has the {\emph{tracial Rokhlin property}} if for every finite set $F\subset A$, every $\ep>0$, and every positive element $x\in A$ with $\|x\|=1$, there are mutually orthogonal projections $e_g\in A$ for $g\in G$  such that:
\begin{enumerate}[(1)]
\item
$\|\af_g(e_h)-e_{gh}\|<\ep$ for all $g, h\in G$.
\item
$\|e_g a-ae_g\|<\ep$ for all $g\in G$ and all $a\in F.$ 
\item
Let $e=\sum_{g\in G}e_g$, we have $1-e\pa x.$ 
\item
$\|exe\|>1-\ep.$
\end{enumerate}
\end{dfn}

The comparison used in condition (3) is Blackadar's comparison, whose definition is given by:
\begin{dfn}\label{B-C}
Let $a$, $b$ be two positive elements in a \ca\ $A$. We say $a\sim_s b$ if there exist some element $x\in A$, such that $a=xx^*$ and $\Her(x^*x)=\Her(b).$ We say $a\pa b$ if there exist $a^{\prime}\in\Her(b)$ such that $a\sim_s a^{\prime}$.
\end{dfn}
It should be pointed out here that Blackadar's comparison is a generalization of Murray-Von Neumann comparison for projections. The notation is adopted from \cite{ORT}.\newline
In Definition \ref{TRP}, we will call (3) the comparison condition and (4) the norm condition. The projections $e_g$ for $g\in G$ will be called Rokhlin projections corresponds to $F$, $x$, and $\ep$, or Rokhlin projections for short.

\section{Rokhlin properties for induced actions}
Induced actions are important ways to generate interesting examples of actions. We'll study the Rokhlin properties for induced actions in this section. To emphsize the distinction between Rokhlin property and tracial Rokhlin property, we will call Rokhlin property the strict Rokhlin property. There are many ways to induce a new action from the old ones, In most cases,  strict Rokhlin property can be inherited.
\begin{prp}
Let $\af\colon G\rightarrow \Aut(A)$ be an action with the strict Rokhlin property, and let $p$ be an invariant projection. Then the induced action $\af|_{pAp}$ has the strict Rokhlin property. 
\end{prp} 
\begin{proof}
We follow essentially the same lines of Lemma 3.7 of \cite{Ph-F}. Let $F\subset pAp$ be finite, let $\ep>0$. Set $n=\mathrm{Card}(G)$. Using semiprojectivity of $\mathbb{C}^n$, Set
\begin{equation*}
\ep_0=min(\frac{1}{n}. \frac{\ep}{4n+1})
\end{equation*}
choose $\dt>0$ such that whenever $B$ is a unital \ca, $q_1,\cdots,q_n$ are mutually orthogonal projections, and $p\in B$ is a non-zero projection such that $\|pq_j-q_jp\|<\dt$ for $1\leq j\leq n$, then there are mutually orthogonal projections $e_j\in pBp$ such that $\|e_j-pq_jp\|<\ep_0$ for $1\leq j \leq n$. We also require $\dt<\ep_0.$\\
Since $\af$ has strict Rokhlin property, with $F\cup \{p\}$ in place of $F$, with $\dt$ in place of $\ep$, we can obtain projections $q_g\in A$ for $g\in G$.  By the choice of $\dt$, there are mutually orthogonal projections $e_g\in pAp$ such that $\|e_g-pq_gp\|<\ep_0$ for $g\in G$. We now estimate, using $\af_g(p)=p$,
\begin{equation*}
\|\af_g(e_h)-e_{gh}\|\leq \|e_h-pq_hp\|+\|e_{gh}-pq_{gh}p\|+\|p(\af_g(q_h)-q_{gh})p\|<2\ep_o+\dt\leq \ep.
\end{equation*}
And for $a\in F$, using $pa=ap=a$,
\begin{equation*}
\|e_ga-ae_g\|\leq 2\|e_g-pq_gp\|+\|p(q_ga-aq_g)p\|<2\ep_0+\dt\leq \ep
\end{equation*}
Next, set $e=\sum_{g\in G}e_g$, then $\|e-p\|\leq\|e_g-pq_gp\|<n\ep_o\leq 1$. But $e$ is also a subprojection of $p$, this force $e=p$, the identity of $pAp.$
\end{proof}

\begin{prp}\label{SRP-Lim}
  Let $\af\colon G\to \Aut(A)$ be an inductive limit action, i.e.\, there exists an increasing sequence $(A_n)_{1<n<\infty}$ such that $A=\overline{\cup_{n\in \N}A_n}$ and each $A_n$ is invariant under the action. Let $\af_n$ denote the induced action on $A_n$. If each $\af_n$ has the strict Rokhlin property, then $\af$ has the strict Rokhlin property. 
\end{prp}
\begin{proof}
Let $F$ be a finite subset of $A$, and $\ep>0$. We can then find some $A_n$ and a finite subset $\tilde{F}$ of $A$, such that for any $a\in F$, there is some $b$ in $\tilde{F}$ with $\|a-b\|<\ep/3$. Since $\af_n$ has strict Rokhlin property, with $\tilde{F}$ and $\ep/3$, we can obtain projections $\{e_g\}_{e\in G}$ as in Definition \ref{SRP}. For any $a\in F$,   find $b$ in $\tilde{F}$ such that $\|a-b\|<\ep/3$, then
\begin{equation*}
\|ae_g-e_ga\|\leq \|(a-b)e_g\|+\|be_g-e_gb\|+\|e_g(a-b)\|<\ep/3+\ep/3+\ep/3=\ep.
\end{equation*}
\end{proof}

Throughout this paper, $A\otimes B$ will always denote the minimal tensor product of $A$ and $B$. 
\begin{prp}\label{SRP-T}
Let $\af\colon G\rightarrow \Aut(A)$ be an action with the strict Rokhlin property and $\beta\colon G\rightarrow \Aut(B)$ be an arbitrary action. Then the tensor product $\gamma=\af\otimes \beta\colon G \rightarrow \Aut(A\otimes B)$ has the strict Rokhlin porperty.
\end{prp}
\begin{proof}
Let $F\subset A$ be finite and $\ep>0$, without loss of generality we may assume that the elements of $F$ are all elementary tensors, i.e.\ $F=\{a_i\otimes b_i\}$. Let $\tilde{F}=\{a_i\}\subset A$ and let $M=max\{\|b_i\|\}$. From the definition \ref{SRP}, we can obtain Rokhlin projections $\{e_g\}_{g\in G}\subset A$ corresponds to $\tilde{F}$ and $\ep/(M+1)$. Consider projections $p_g=e_g\otimes 1\in A\otimes B$, they are mutually orthogonal projections sum up to 1, and satisfy:
\begin{enumerate}[(1)]
\item
$\|\af_g(p_h)-p_{gh}\|=\|(\af_g(e_h)-e_{gh})\otimes 1\|<\ep$.
\item
$\|p_g(a_i\otimes b_i)-(a_i\otimes b_i)p_g\|=\|(e_ga_i-a_ie_g)\otimes b_i\|<\ep/(M+1)M<\ep$.
\end{enumerate}
\end{proof}

Now let's turn to tracial Rokhlin property. A \ca\ is said to have {\emph{Property (SP)}} if every non-zero hereditary C*-subalgebra contains at least one non-zero projection. The following observation is easy to see:
\begin{lem}\label{SP-SRP}
(Lemma 1.13, \cite{Ph-F}) Let $A$ be an infinite dimensional simple separable unital \ca, let $\af \colon G\rightarrow \Aut(A)$ be an action of a finite group $G$ on $A$ which has the tracial Rokhlin property. Then $A$ has Property (SP) or $\af$ has the strict Rokhlin property. 
\end{lem}
Not surprisingly we will mostly deal with C*-algebras with Property (SP) when studying tracial Rokhlin property. So we present some facts about Property (SP) here.
\begin{dfn}\label{Fun1}
Let $0\leq\sm_2<\sm_1\leq 1$ be two positive numbers. The function $f^{\sm_1}_{\sm_2}$ is a piecewise linear function defined by:
\begin{equation*}
f^{\sm_1}_{\sm_2}(x)=
\begin{cases}
0 & \text{if\quad} x<\sm_2\\
\frac{x-\sm_2}{\sm_1-\sm_2} & \text{if\quad} \sm_2\leq x \leq \sm_1\\
1 & \text{if\quad} x>\sm_1
\end{cases}
\end{equation*}
\end{dfn}

\begin{lem}
  \label{SP}
Let $A$ be a \ca\ with Property (SP), let $x\in A$ be a positive element with norm 1, and let $\ep>0$. Then there exists a non-zero projection $p\in \overline{xAx}$ such that for any positive element $q\leq p$ with$ \|q\|=1$, we have:
\begin{equation*}
\|qx^{1/2}-q\|\leq\ep, \quad \|x^{1/2}q-q\|\leq\ep.
\end{equation*}
\end{lem}
\begin{proof}
(See Lemma 1.14 of ~\cite{Ph-F}.) Since $\Her(x)=\Her(x^{1/2})$, it suffices to prove the inequalities with $x$ in place of $x^{1/2}.$  \Wolog\ Assume that $\ep<1/2$. Let $f=f_{0}^{\ep}$ ,$g=f_{1-\ep}^1$. Then we have:
\begin{equation*}
fg=g,\quad \|f(x)-x\|\leq \ep,\quad \overline{g(x)Ag(x)}\subset \overline{xAx}=\overline{f(x)Af(x)}. 
\end{equation*}
Now $fg=g$ implies that for any $c\in \overline{g(x)Ag(x)}$, $cf(x)=f(x)c=c$. Since $A$ has Property (SP), we can choose a nonzero projection $p\in \overline{g(x)Ag(x)}\subset \overline{xAx}$. Then for any positive element $q\leq p$,  $\|qx-q\|=\|qx-qf(x)\|\leq \ep$ and similarly $\|xq-q\|\leq \ep$. 
\end{proof}

A C*-algebra is elementary if it's a finite dimensional C*-algebra or the C*-algebra of compact operators.
\begin{lem}
 \label{SP2}
Let $A$ be a non-elementary simple C*-algebra with Property (SP). Then for any non-zero projection $p$ in $A$ and any $n\in \N$, there exists $n$ mutually orthogonal sub-projections $p_i$ of $p$ which are mutually Murray-von Neumann equivalent. 
\end{lem}
\begin{proof}
 Lemma 3.5.7, p142 of ~\cite{Lin-C} 
\end{proof}

\begin{lem}\label{SP-Lim}
If there exists an increasing sequence $\{A_n\}$ such that\\
 $A=\overline{\cup_{n\in \N}A_n}$ and each $A_n$ has Property (SP), then $A$ itself has Property (SP).
\end{lem}
\begin{proof}
(Private communication with C. Philips) Let $x\in A$ be a positive element of norm 1. We need to find a non-zero projection in $\overline{xAx}$. Let $\ep=1/28.$  Then we can find some $n\in \mathbb{N^{+}}$ and a positive element $y\in A_n$ such that $\|y-x\|<\ep$. This implies $y-\ep<x$. Let $z$ be the positive part of $y-\ep$. We have:
\begin{equation*}
z\neq 0, z(y-\ep)z=z^3,\quad\|z-(y-\ep)\|<2\ep \quad\text{and}\quad \|z-x\|<4\ep.
\end{equation*}
Let $c=zx^{1/2}$, then $1-4\ep< \|c\|\leq 1$. Since $cc^*=zxz\geq z(y-\ep)z\geq z^3$, we have:
\begin{equation*}
 \overline{cc^*Acc^*}\supset \overline{z^3Az^3}\supset \overline{zAz}\supset \overline{zA_nz}.
\end{equation*}
Let $f=f_{\ep}^{1-\ep}$ and $g=f_{1-\ep}^1.$ (See Definition \ref{Fun1}.) Since $A_n$ has Property (SP) by assumption, we can find non-zero projection $p\in \overline{g(z)A_ng(z)}\subset\overline{cc^*Acc^*}$. Now $f(z)g(z)=g(z)$ implies that $pf(z)=pg(z)f(z)=p$. Use the relation $\|z-x\|<4\ep$ and $\|pz-p\|=\|pz-pf(z)\|<\ep$ repeatedly, we can show that $\|pcc^*-p\|=\|pzxz-p\|<7\ep$.  Take the adjoint we get $\|cc^*p-p\|<7\ep$. Then $c^*pc$ is approximately a projection: $\|c^*pcc^*pc-c^*pc\|<7\ep<1/4$. Since $c^*pc=x^{1/2}zpzx^{1/2}\in \overline{xAx}$, there exist a projection $q\in \overline{xAx}$ such that $\|q-c^*pc\|<1/4$. The projection $q$ is non-zero, otherwise
\begin{equation*}
\|c^*pc\|=\|pcc^*p\|\geq\|p\|-\|(pcc^*-p)p\|=3/4>1/4>\|c^*pc\|.
\end{equation*} 
Which is a contradiction. 
\end{proof}

The following lemma is very useful in dealing with tensor products:
\begin{lem}
 \label{Slice}
(Kirchberg's Slice Lemma). Let $A$ and $B$ be C*-algebras, and let $D$ be a non-zero hereditary sub-C*-algebra of the minimal tensor product $A\otimes B$. Then there exists a non-zero element $z$ in $A\otimes B$ such that $z^*z$ is an elementary tensor $a\otimes b$, for some $a\in A_{+}$ and $b\in B_{+}$, and $zz^*$ belongs to $D$.
\end{lem}
See lemma 4.1.9, p68 of  book ~\cite{RS1} for a proof. Note that the definition of $a$, $b$ in the proof shows that they are all positive. As a consequence, we can show the following:
\begin{lem}\label{SPtensor}
Let $A$,$B$ be two C*-algebras with Property (SP), then $A \otimes B$ has Property (SP).
\end{lem}
\begin{proof}
Let $D$ be a non-zero hereditary sub-C*-algebra of $A\otimes B$. By Lemma \ref{Slice}, we can find a non-zero element $z$ in $A\otimes B$ such that $z^*z$ is an elementary tensor $a\otimes b$ for some $a\in A_{+}$ and $b\in B_{+}$, and $zz^*\in D$. Since both $A$ and $B$ have Property (SP), there exist non-zero projections $p$, $q$ in $\Her(a)$ and $\Her(b)$, respectively. Then $p\otimes q$ is a non-zero projection in $\Her(a\otimes b)=Her(z^*z)$. But $\Her(z^*z)$ is isomorphic to $\Her(zz^*)\subset D$ (See p 218 of \cite{Cu1}), therefore $D$ contains a non-zero projection.
\end{proof}

Before we systematically study the inheritance of tracial Rokhlin property, let's first present some basic properties of Blackdar's comparison first. (See Definition \ref{B-C}. We have the following equivalent definition:
\begin{lem}
Let $a$, $b$ be two positive elements in a \ca\ $A$. Let $A''$ be the enveloping von-Neumann algebra of $A$. Then $a\pa b$ \ifo\ there exist some partial isometry  $v\in A''$ such that, $v\Her(a)$, $\Her(a)v$ are subsets of $A$, $vv^*=p_{a}$, where $p_{a}$ is the range projection of $a$ in $A''$, and $v^*\Her(a)v\subset \Her(b)$. $a\sim_s b$ \ifo\ $v^*\Her(a)v=\Her(b)$.
\end{lem}
See Proposition 4.3 and Proposition 4.6 of \cite{ORT} for a proof.
\begin{ntn}
By $a_1\oplus a_2\oplus\cdots\oplus a_n$ we mean $\diag\{a_1,a_2,..., a_n\}$ and $n\odot a$ means $a\oplus a\oplus\cdots\oplus a.$
We write $a_1\oplus a_2\oplus\cdots\oplus a_n\leq b_1\oplus b_2\oplus\cdots\oplus b_m$ if and only if $\diag\{a_1,a_2,..., a_n,0,..0\}]\pa\diag\{b_1,b_2,..., b_m, 0,...,0\}$ in $M_{m+n}(A).$ 
\end{ntn}

\begin{prp}(Proposition 3.5.3, p 141,\cite{Lin-C}) Let $A$ be a C*-algebra.
\begin{enumerate}[(i)]
\item
If $0\leq a\leq b$, then $a\pa b.$
\item
If $p$ and $q$ are two projections in $A$, then $p\pa q$ \ifo\ $p$ is sub-equivalent to $q$ in the sense of Murray and Von Neumann, and $p\sim_s q$ if and only if $p$ and $q$ are Murray-von Neumann equivalent.
\item
Let $B$ be a hereditary subalgebra of $A$ and $a,b\in A$. Then $a\pa b$ in $A$ if and only if $a\pa b$ in $B$, and  $a\sim_s b$ in $A$ if and only if $a\sim_s b$ in $B$.
\item
Let $a_1, ..., a_n$ be positive elements in $A$. Then $(\sum_{i=1}^n{a_i})\pa \oplus_{i=1}^n {a_i}.$ If $a_ia_j=0,\forall i\neq j$, then $(\sum_{i=1}^n{a_i})\sim_s\oplus_{i=1}^n {a_i}$
\end{enumerate}
\end{prp}

Now we are ready to prove analogue results for tracial Rokhlin property:
\begin{prp}
Let $\af\colon G\rightarrow \Aut(A)$ be an action with the tracial Rokhlin property, and let $p$ be an invariant projection. Then the induced action $\af|_{pAp}$ has the tracial Rokhlin property. 
\end{prp} 
\begin{proof}
Lemma 3.7 of \cite{Ph-F}
\end{proof}

\begin{prp}
  Let $\af\colon G\to \Aut(A)$ be an inductive limit action (See Proposition \ref{SRP-Lim} for the definition.) Let $\af_n$ denote the induced action on $A_n$. If each $\af_n$ has the tracial Rokhlin property, then $\af$ has the tracial Rokhlin property. 
\end{prp}
\begin{proof}
 If there are infinite many $A_n$'s that do not have Property (SP),  then $\af_n$ has the strict Rokhlin property for infinite many $n$, so $\af$ also has the strict Rokhlin property. (A slight modification of Proposition \ref{SRP-Lim}).\\
Otherwise, A slight modification of Lemma \ref{SP-Lim} shows that $A$ itself has Property (SP). Let $F$ be a finite subset of $A$. Let $x\in A_{+}$ with $\|x\|=1$, and $\ep>0$ be given.\\
Since $A$ has Property (SP), by Lemma ~\ref{SP} we can find a non-zero projection $p$ such that $\|px^{1/2}-p\|<\ep$ and $\|x^{1/2}p-p\|<\ep$. We can then find some $n\in \N$ and a non-zero projection $q\in A_n$ such that $\|q-p\|<\ep$. Without loss of generality, we may also assume that $F\subset A_n$. Since $\af_n$ has the tracial Rokhlin property, let $\dt=min\{\ep/7, 1/2\}$, we can then find mutually orthogonal projections $\{e_g\}_{g\in G}\subset A_n$ such that:
\begin{enumerate}[(1)]
\item
$\|e_ga-ae_g\|<\dt.$
\item
 $\|\af_g(e_h)-e_{gh}\|<\dt.$
\item
Let $e=\sum_{g\in G}e_g$, then $1-e\pa q$
\item
 $\|eqe\|>1-\dt$
\end{enumerate}

 Since $\|p-q\|<\dt\leq 1/2,$ $p$ and $q$ are Murray-von Neumann equivalent in $A$. But $p\in \overline{xAx}$, so $1-e\pa q\approx p\pa x$.
For the norm condition, we have following estimation:
\begin{eqnarray*}
\|exe\|&=&\|ex^{1/2}x^{1/2}e\|=\|x^{1/2}ex^{1/2}\|\geq\|px^{1/2}ex^{1/2}p\|\\
       &=&\|(px^{1/2}p-p)(x^{1/2}p-p)+pe(x^{1/2}-p)+(px^{1/2}-p)ep+pep\|\\
       &\geq& \|pep\|-3\dt\\
       &\geq&\|qeq\|-3\dt-3\dt\\
       &=&\|eqe\|-6\dt>1-7\dt\geq 1-\ep.
\end{eqnarray*}
\end{proof}

\begin{prp}
Let $\af \colon G\rightarrow \Aut(A)$ be an action of finite group $G$ on a simple \ca\ A with the tracial Rokhlin property. Let $\beta \colon G\rightarrow \Aut(B)$ be an arbitrary action on a simple \ca\ B. Let $\theta=\af\otimes \beta \colon G\rightarrow \Aut(A\otimes B)$ be the tensor action of $\af$ and $\beta$. If $A\otimes B$ has Property (SP), then $\theta$ has the tracial Rokhlin property.
\end{prp}
\begin{proof}
First of all, we can assume that $A$ has Property (SP). Otherwise the action $\af$ will have the strict Rokhlin property and therefore $\theta$ has the strict Rokhlin property by Proposition \ref{SRP-T}.\\

Let $F$ be a finite subset of $A\otimes B$, $\ep>0$, $x\in A\otimes B$ be a positive element of norm 1.  Without loss of generality, we may assume that $F$ consists of elementary tensors $a_i\otimes b_i, 1\leq i\leq n$. \\

Let $\ep_0=\ep/12$. Since $A\otimes B$ has Property (SP), by Lemma \ref{SP}, we can find a non-zero projection $r\in Her(x)$ such that for any projection $s\leq r$, we have $\|sx^{1/2}-s\|\leq\ep_0, \|x^{1/2}s-s\|\leq\ep_0$. By Kirchberg's Slice Lemma \ref{Slice}, there exist some $z\in A\otimes B$, such that $zz^*\in Her(r)$, and $z^*z=a\otimes b$, for some $a\in A_+$ and $b\in B_+$. We may assume that $\|a\|=\|b\|=\|z\|=1$. Find $z_0=\sum_{j=1}^{k} {y_j\otimes z_j}$ with norm 1 such that $\|z-z_0\|<\ep_0$. \newline

Since $B$ is simple, we can find elements $\{l_i |1\leq i\leq m\}$ such that $\sum_{i}{l_i b l_i^*}$=1.  Let $M=max\{\|a_i\|,\|b_i\|,\|y_j\|,\|z_j\|\}$ and let $\dt=\frac{\ep}{2(|G|)M}$. For $a\in A$, we first use Lemma \ref{SP}, find a non-zero projection $p\in Her(a)$ such that for any projection $q\leq p$, $\|qa^{1/2}-q\|\leq\dt, \|a^{1/2}q-q\|\leq\dt$. By Lemma \ref{SP2} we can find $m$ mutually orthogonal but equivalent subprojection $\{p_i\}_{1\leq i\leq m}$ of $p$.\newline

Since $\af$ has the tracial Rokhlin property, for $F^{\prime}=\{a_i\}\cup \{y_k\}$, $p_1\in A_+$ and $\dt>0$ as chosen before, we can find projections $\{q_g\}_{g\in G}$ in $A$ such that:
\begin{enumerate}[(1)]
\item
$\|q_g d-d q_g\|<\dt$, $\forall d\in F^{\prime} $ and $g\in G$
\item
$\|\af_g(q_h)-q_{gh}\|<\dt$.
\item
With $q=\sum_{g\in G}q_g$, $1-q\pa p_1$
\item
$\|qp_1q\|\geq 1-\dt$
\end{enumerate}
Now consider the projections $e_g=q_g\otimes 1$. For the action $\theta$, we have:
\begin{enumerate}[($1^{\prime}$)]
\item
$\|e_g f-f e_g\|=\|(q_ga_i-a_iq_g)\otimes b_i\|<M\dt\leq \ep$, $\forall a_i\otimes b_i\in F$ and $g\in G$
\item
$\|\theta_g(e_h)-e_{gh}\|=\|(\af_g(q_h)-q_{gh})\otimes 1\|<\dt\leq \ep$
\item
Let $e=\sum_{g\in G}e_g=q\otimes 1$, 
\begin{align*}
 1-e&=(1-q)\otimes (\sum{l_i b l_i^*})\\
&=\sum_{i=1}^{m}(1\otimes l_i)((1-q)\otimes b)(1\otimes l_i^*\\
&\pa m\odot (1-q)\otimes (b\pa m)\odot (p_1\otimes b)\\
&\pa a\otimes b\sim_s z^*z\pa x\\
\end{align*}
\item
Since $p_1\leq p$, by our choice of $p$, we have $\|p_1a^{1/2}-p_1\|\leq\dt$ and  $\|a^{1/2}p_1-p_1\|\leq\dt$. Therefore
\begin{equation*}
 \|qaq\|=\|a^{1/2}qa^{1/2}\|\geq\|p_1a^{1/2}qa^{1/2}p_1\|>\|p_1qp_1\|-2\dt=\|qp_1q\|-2\dt>1-3\dt.
\end{equation*}
It follows that $\|ezz^*e\|=\|e(a\otimes b)e\|=\|(qaq)\otimes b\|>1-3\dt$. Hence
\begin{align*} 
\|ez^*ze\|&>\|ez_0^*z_0e\|-2\ep_0\\
&>\|z_0^*ez_0\|-2M|G|\dt-2\ep_0\\
&>\|z^*ez\|-2M|G|\dt-4\ep_0\\
&=\|ezz^*e\|-2M|G|\dt-4\ep_0 
\end{align*}
Recall that our $r\in \Her(x)$ is chosen so that  $\forall s\leq r$, $\|sx^{1/2}-s\|\leq\ep_0, \|x^{1/2}s-s\|\leq\ep_0$. Since $zz^*\in \Her(r)$, we have $\|zz^*x^{1/2}-zz^*\|=\|zz^*(rx^{1/2}-r)\|\leq \dt.$ Hence 
\begin{align*}
\|exe\|&>\|(zz^*)^{1/2}x^{1/2}ex^{1/2}(zz^*)^{1/2}\|\\
&>\|(zz^*)^{1/2}e(zz^*)^{1/2}\|-2\ep_0\\
&=\|ezz^*e\|-2\ep_0>1-2M|G|\dt-6\ep_0\geq 1-\ep.\\
\end{align*}
\end{enumerate}
\end{proof}
By Lemma \ref{SPtensor}, we have the following corollary:
\begin{cor}
Let $\af \colon G\rightarrow \Aut(A)$ be an action of finite group $G$ on a simple \ca\ $A$, with the tracial Rokhlin property. Let $\beta \colon G\rightarrow \Aut(B)$ be an arbitrary action on simple \ca\ $B$. Let $\theta=\af\otimes \beta \colon G\rightarrow \Aut(A\otimes B)$ be the tensor product of $\af$ and $\beta$. If $B$ has Property (SP), then $\theta$ has the tracial Rokhlin property.
\end{cor}

\section{Rokhlin Properties for product-type actions}
In this section, we give a complete classification or Rokhlin properties for product-type actions. 
\begin{dfn}\label{ptact}
Let $A=\otimes_{i=1}^{\infty}\mathrm{B}(H_i),$ where $H_i$ is a finite dimensional Hilbert space for each $i$. Let $G$ be a finite group. An action $\af\colon G\mapsto \Aut(A)$ is called a {\emph{product-type action}} \ifo\ for each $i$, there exists a unitary representation $\pi_i\colon G\rightarrow \mathrm{B}(H_i)$, which induces an inner action $\af_i\colon g\mapsto \Ad(\pi_i(g))$, such that $\af=\otimes_{i=1}^{\infty}\af_i$.
\end{dfn}
\begin{dfn}
Let $\af\colon G\mapsto \Aut(A)$ be a product-type action on a UHF-algebra $A.$ A \emph{telescope} of the action is a choice of an infinite sequence of positive integers $1=n_1<n_2<\cdots$ and a re-expression of the action, so that $A=\otimes_{i=1}^{\infty} B(K_i)$ where $K_i=\otimes_{j=n_i}^{n_{i+1}-1}H_j$, and the action on $B(K_i)$ is $\otimes_{j=n_i}^{n_{i+1}-1}\af_j$
\end{dfn}

Recall that two actions $\af\colon G\mapsto \Aut(A)$ and $\bt \colon G\mapsto \Aut(B)$ are said to be conjugate, if there exist an isomorphism $T\colon A\mapsto B$ such that $T\circ \af_g=\bt_g\circ T$, for any $g\in G$. The main result of this section is the the following theorem:
\begin{thm}\label{main}
Let $\af \colon G\mapsto \Aut(A)$ be a product-type action where $A$ is UHF. Let $H_i$,$\pi_i$,$\af_i$ be defined as in Definition \ref{ptact}. Let $d_i$ be the dimension of $H_i$ and $\chi_i$ be the character of $\pi_i$. We will use the same notations if we do a telescope to the action. Define $\chi\colon G\mapsto \C$ to be the characteristic function on $1_{G}$. Then we have:
\begin{enumerate}[(i)]
\item
$\af$ has strict Rokhlin property \ifo\ there exists a telescope, such that for any $n\in \N$,  
\begin{equation}
\frac{1}{d_n}\chi_n=\chi
\end{equation}
\item
$\af$ has the tracial Rokhlin property \ifo\ there exists a telescope, such that for any $n\in \N$, the infinite product
\begin{equation}
{\displaystyle{\prod_{n\leq i <\infty}\frac{1}{d_i}\chi_i}}=\chi
\end{equation}
\end{enumerate}
\end{thm}

In \cite{Ph-C}, Phillips gave a complete characterization of Rokhlin actions in terms of the unitaries, when $G$ is $\Z/2\Z$. This theorem is a generalization of Proposition 2.4 and Proposition 2.5 of \cite{Ph-C}. We shall deal with finite cyclic group first, and then extend the result to arbitrary finite group. \newline

We first develop some terminology for cyclic groups as well as product-type actions.
\begin{lem}
Let $G$ be any finite cyclic group, then there exist a commutative bihomomorphism $\zt \colon G\times G\rightarrow S^1$, where $S^1$ is the unit circle of the complex plane. By a commutative bihomomorphism we mean that for any 
$g,h\in G,$ $\zt(\bullet ,g)$ and $\zt(h,\bullet)$ are homomorphisms and $\zt(g,h)=\zt(h,g).$
\end{lem}
\begin{proof}
Let $G$ be a cyclic group of order $n$, which is generated by a primitive $n$-th root of unity $\xi$. Define $\zt : G\times G
\rightarrow S^1$ by $\zt(\xi^i,\xi^j)=\xi^{ij}.$ Computation shows that $\zt$ is a commutative bihomomorphism.
\end{proof}
\begin{rmk}
From now on we will assume that a commutative bihomomorphism $\zt$ is given whenever we have a finite cyclic group $G$. It's easy to see that the map $g\rightarrow \zt(g,\bullet)$ defines an isomorphism between $G$ and the dual group $\hat{G}$. We use $\zt^g(h)$ to denote $\zt(g,h)$ as an indication of this duality. It's not difficult to see that $\sum_{g\in G}\zt^g(h)=\dt(h,1_{G})|G|$, where $\dt(\cdot,\cdot)$ is the Kronecker delta. This equality will be used in later computations.
\end{rmk}

Product-type actions are special cases of inductive limit actions. By Theorem 4.5 of \cite{Ph-S}, we have 
\begin{equation*}
C^*(G,A,\af)=\varinjlim C^*(G,A_n,\af)
\end{equation*}

The action on each $A_n$ is inner, by Example 4.10 of \cite{Ph-S}, $C^*(G,A_n,\af)\cong C^*(G)\otimes A_n$. If $G$ is abelian, then $C^*(G)\cong C(G)$. In particular, we have the following: 
\begin{prp}\label{ISO_C}
Let $G$ be a finite abelian group. Then $C^*(G,A_n,\af)$ is isomorphic to a direct sum of $|G|$ copies of $A_n$. We may write $C^*(G,A_n,\af)\cong \oplus_{g\in G}A_n^g$, where $A_n^g=A_n$ for all $g\in G$.
\end{prp}

Now let's find an explicit formula for the isomorphism in the above proposition as well as the corresponding direct system.\\

For any $g\in G$, Let $V_n^g \in U(A_n)$ denote the finite tensor $\otimes_{i=1}^{n}\pi_i(g)$. Define $X_{n}^g=(\zt^h(g)V_{n}^g)_{h\in G}\in \oplus_{g\in G} A_n^g.$ Embed $A_n$ into $\oplus_{g\in G}A_n^g$ by the map $a\mapsto (a^h)_{h\in G}$, where $a^h=a$ for all $h\in G$. We can check that $\{X_{n}^g\}_{g\in G}$ are unitary elements and $A_n\cup\{X_{n}^g\}_{g\in G}$ satisfy the relation:
\begin{equation*}
X_{n}^{g}X_{n}^{h}=X_{n}^{gh},\quad X_{n}^{g}aX_{n}^{g*}=\af_g(a),\quad \forall g,h\in G, a\in A.
\end{equation*}
Let $U_{n}^g$ be the canonical unitaries of $C^*(G,A_n,\af)$. By Lemma 3.17 of \cite{Ph-S}, there is a *-homomorphism  $T\colon C^*(G,A_n,\af) \mapsto \oplus_{g\in G}A_n^g$ which sends $U_g$ to $X_g$ and maps $A_n$ to $\oplus_{g\in G}A_n^g$ the same way we embeded $A_n$ in $\oplus_{g\in G}A_n^g$. Now that this homomorphism is an isomorphism is equivalent to that the matrix $(\zt^h(g))_{h,g\in G}$ is invertible. But $(\zt^h(g))_{h,g\in G}$ is actually unitary(after normalization): For any two elements $h_0,h\in G$,
\begin{eqnarray*}
\sum_{g\in G}{\zt^{h_0}(g)\overline{\zt^h(g)}}
&=&\sum_{g\in G}{\zt^{h_0}(g)\zt^{h^{-1}}(g)}=\sum_{g\in G}{\zt^{h_0h^{-1}}(g)}\\
&=&\dt(h_0h^{-1},1)|G|=\dt(h_0,h)|G|
\end{eqnarray*}

For each n, we can define an isomorphism as above. This enables us to turn the direct system ${\varinjlim } C^{\ast}(G,A_n,\af)$ into a direct system ${\varinjlim } \oplus_{g\in G}A_n^g$. Let $\phi_{i,j}$ be the connecting map of the latter system, we have the following proposition:
\begin{prp}\label{C-Map}
For any positive integers $n\leq m$ and any $g\in G$, let $P_{n,m}^g$ be an element of $\otimes_{i=n}^{m}\B(H_i)$ defined by:
\begin{equation*}
P_{n,m}^g=|G|^{-1}\sum_{h\in G}{\zt^h(g)\otimes_{i=n}^{m}\pi_i(h)}
\end{equation*}
Then $P_{n,m}^g$ are \mops\ sum up to the 1 such that the connecting map $\phi_{n-1,m}$ is given by:
\begin{equation*}
 (y_{n-1}^h)_{h\in G}\mapsto (y_m^h)_{h\in G}, \quad \text{where}\,\, y_m^h=\sum_{g\in G}{y_{n-1}^g\otimes P_{n,m}^{hg^{-1}}}.   
\end{equation*}
We shall use $P_n^g$ to denote $P_{n,n}^g$
\end{prp}
\begin{proof}
\Wolog\ we may assume $m=n$. Let $P_{n}^g=|G|^{-1}\sum_{h\in G}{\zt^h(g)\pi_n(h)}$. We can compute:
\begin{eqnarray*}
(P_n^g)^*&=&|G|^{-1}\sum_{h\in G}{\overline{\zt^h(g)}(\pi_n(h))^*}\\
&=&|G|^{-1}\sum_{h\in G}{\zt^{h^{-1}}(g)\pi_n(h^{-1})}=P_n^g.\\
P_n^{g_0}P_n^g&=&|G|^{-2}\sum_{h,k\in G}{\zt^h(g_0)\zt^k(g)\pi_n(h)\pi_n(k)}\\
&=&|G|^{-2}\sum_{h,k\in G}{\zt^h(g_0g^{-1})\zt^{hk}(g)\pi_n(hk)}\\
&=&|G|^{-2}\sum_{h\in G}{\zt^h(g_0g^{-1})|G|P_n^g}\\
&=&|G|^{-1}\dt(g_0,g)|G|P_n^g=\dt(g_0,g)P_n^g\\
\sum_{g\in G}{P_n^g}&=&|G|^{-1}\sum_{h,g\in G}{\zt^h(g)\pi_n(h)}\\
&=&|G|^{-1}\sum_{h\in G}{\dt(h,1_{G})|G|\pi_n(h)}=\pi_n(1_{G})=I
\end{eqnarray*}
The above computation shows that $\{P_{n}^g\}_{g\in G}$ are mutually orthogonal projections sum up to 1. Define a *-homomorphism $\phi_{n-1,n}:\oplus_{g\in G}A_{n-1}^g \rightarrow \oplus_{g\in G}A_n^g$ by:
\begin{equation*}
 (y_{n-1}^h)_{h\in G}\mapsto (y_n^h)_{h\in G}, \quad \text{where}\,\,
 y_n^h=\sum_{g\in G}{y_{n-1}^g\otimes P_n^{hg^{-1}}}.
\end{equation*}
 We shall verify that $\phi_{n-1,n}$ is the induced connecting map on $\varinjlim  \oplus_{g\in G}A_n^g$.
 The connecting map on the crossed product is determined by $a\mapsto a\otimes 1$ and $U_{n-1}^g\mapsto U_{n}^g$. Adopt notation in proposition \ref{ISO_C}, let $T\colon C^*(G,A_n,\af)\mapsto \oplus_gA_n^{g}$ be the isomorphism such that $T(a)=(a,\dots,a)$ and $T(U_{n-1}^g)=X_{n-1}^g=(\zt^h(g)V_{n-1}^g)_{h\in G}$. For any $l\in G$, let $\phi_{n-1,n}^l$ be the composition of $\phi_{n-1,n}$ and the natural projection $\pi^l:\oplus_{g\in G}A_n^g\rightarrow A_n^l$. We have: 
\begin{eqnarray*}
 \phi_{n-1,n}^l((a^h)_{g\in G})&=&\sum_{g\in G}{a\otimes P_n^{lg^{-1}}}\\
  &=& a\otimes\sum_{g\in G}{P_n^{lg^{-1}}}=a\otimes 1\\
 \phi_{n-1,n}^l(X_{n-1}^g)&=&\sum_{h\in G}{\zt^h(g)V_{n-1}^g\otimes P_n^{lh^{-1}}}\\
 &=&\sum_{h\in G}{\zt^h(g)V_{1,n-1}^g\otimes |G|^{-1}\sum_{k\in G}{\zt^k(lh^{-1})\pi_n(k)}}\\
 &=&|G|^{-1}\sum_{h,k\in G}{\zt^h(g)\zt^k(lh^{-1})V_{n-1}^g\otimes \pi_n(k)}\\
 &=&|G|^{-1}\sum_{h,k\in G}{\zt^h(g)\zt^k(lh^{-1})V_{n-1}^g\otimes \pi_n(k)}\\
 &=&|G|^{-1}\sum_{h,k\in G}{\zt^h(gk^{-1})\zt^l(k)V_{n-1}^g\otimes \pi_n(k)}\\
 &=&|G|^{-1}\sum_{k\in G}{\dt(g,k)|G|\zt^l(k)V_{n-1}^g\otimes \pi_n(k)}\\
 &=&\zt^l(g)V_{n-1}^g\otimes \pi_n(g)=\pi^l(X_{n}^g)
 \end{eqnarray*}
Hence $\phi_{n-1,n}(X_{1,n-1}^g)=X_{1,n}^g$, and the proof is complete.
\end{proof}

\begin{rmk}\label{C-MapR}
We can get our unitaries back from the projections constructed in the above proof as follows: $\pi_n(g)=\sum_{l\in G}{{\overline{\zt^{g^{-1}}(l)}}P_n^l}$. We could also observe that the projections $P_n^g$ are invariant projections, but some of them may be 0.
\end{rmk}

Suppose the group $G$ is finite cyclic, recall that $g\mapsto \zt^g(\bullet)$ gives an isomorphism between G and its dual. The dual action on the crossed product turns out to have a very simple description under this isomorphism:
\begin{prp}\label{DualAction}
We identify the crossed product $C^{\ast}(G,A,\af)$ with the direct system $\varinjlim \oplus_{g \in G}A_n^g$, and identify $\hat{G}$ with G using the bihomomorphism $\zt$. Then the dual action $\hat{\af}$, when restrict to $\oplus_{g \in G}A_n^g$, is given by:
\begin{equation*}
 \hat{\af}_g((y^h)_{h\in G})=(z^h)_{h\in G}, \quad \text{where}\,\, z^h=y^{hg}, \forall h\in G.
\end{equation*}

\begin{proof}
The dual action $\hat{\af}$ is defined by $\hat{\af}_{\phi}(f)(t)=\phi(t)f(t)$, for all $\phi \in \hat{G}$ and $f \in C^{\ast}(G,A,\af)$. 
Fix a positive integer n,let $\{U_{n}^g\}$ be the canonical unitaries in $C^{\ast}(G,A_n,\af)$, then $\hat{\af}_h(U_{n}^g)=\zt^h(g)U_{n}^g$. Fix an element $k \in G$, define a map $\beta_k\colon \oplus A_n^g\rightarrow \oplus A_n^g$ by $\beta((y^g)_{g\in G})=(y^{gk})_{g\in G}$. Recall that the isomorphism $T$ between $C^{\ast}(G,A_n,\af)$ and $\oplus A_n^g$ sends $(U_{n}^h)$ to $(\zt^g(h)V_{n}^h)_{g\in G}$. 
Then we the following diagram is commutative:\newline
\begin{math}
\begin{CD}
U_{n}^h @>T>> (\zt^g(h)V_{n}^h)_{g\in G}\\
@VV\hat{\af}_kV     @VV\beta_kV\\
\zt^k(h)U_{n}^h @>T>> (\zt^{gk}(h)U_{n}^h)_{g\in G}\\
\end{CD}
\end{math}

A similar commutative diagram holds for elements in $A_n$. Since these elements generate the corresponding $C^{\ast}$-algebras, we can conclude that $\beta_k$ corresponds to $\hat{\af}_k$ under the isomorphism $T$.
\end{proof} 
\end{prp}

Now we are ready to state our classification result for strict Rokhlin property:
\begin{prp}\label{SRP1} 
 Let $\af\colon G\mapsto \Aut(A)$ be a product-type action, where $G$ is finite cyclic and $A$ is UHF. Then $\af$ has the strict Rokhlin property, \ifo\ up to a telescope, for any $n\in \N$ the projections $P_n^g$ for $g\in G$ constructed in Proposition \ref{C-Map} are mutually Murray-von Neumann equivalent.
\end{prp}
\begin{proof}
  A telescope does not change the action, so let's assume that for any $n\in \N$, the projections $P_{n}^g$ are mutually equivalent. Let $F$ be a finite set in $A$. \Wolog\ we may assume that F is in $A_{n-1}$, for some $n\in \N$. Recall that ${P}_{n}^g=|G|^{-1}\sum_{h\in G}\zt^h(g)\pi_{n}(h)$. Since these projections are equivalent, for each $g\in G$, there exists a partial isometry $W^{1,g}$ such that $W^{1,g}W^{1,g}*=P_{n}^1$ and $W^{1,g}*W^{1,g}=P_{n}^g$ ((Here 1=$1_{G}$ is the identity element). Let $W^{g,1}$ denote the conjugate of $W^{1,g}$ and let $W^{g,h}=W^{g,1}W^{1,h}$. Let $\dt(\cdot,\cdot)$ be the Kronecker delta, we have:
\begin{equation*}
W^{g,g}=P_n^g,\quad W^{g,h}W^{k,l}=\dt(h,k)W^{g,l},\,\, \forall g,h,k,l\in G.  
\end{equation*}
 Let $Q^k=\frac{1}{|G|}\sum_{g,h\in G}\zt^k(g^{-1}h)W^{g,h}$, for any $k$ in $G$. Now we can compute that: 
\begin{align*}
Q^kQ^j&=\frac{1}{|G|^2}(\sum_{g_1,h_1\in G}\zt^k(g_1^{-1}h_1)W^{g_1,h_1})(\sum_{g_2,h_2\in G}\zt^j(g_2^{-1}h_2)W^{g_2,h_2})\\
&=\frac{1}{|G|^2}\sum_{g_1,h_1,h_2\in G}\zt^k(g_1^{-1}h_1)\zt^j(h_1^{-1}h_2)W^{g_1,h_2}\\
&=\frac{1}{|G|^2}\sum_{g_1,h_1,h_2\in G}\zt^k(g_1^{-1})\zt^{kj^{-1}}(h_1)\zt^j(h_2)W^{g_1,h_2}\\
&=\dt(k,j)\frac{1}{|G|}\sum_{g_1,h_2\in G}\zt^k(g_1^{-1})\zt^j(h_2)W^{g_1,h_2}\\
&=\dt(k,j)\frac{1}{|G|}\sum_{g_1,h_2\in G}\zt^k(g_1h_2^{-1})W^{g_1,h_2}=\dt(k,j)Q^k\\
\end{align*}
\begin{align*}
(Q^k)^*&=\frac{1}{|G|}\sum_{g,h\in G}\zt^k(g^{-1}h)(W^{g,h})^*\\
     &=\frac{1}{|G|}\sum_{g,h\in G}\overline{\zt^k(g^{-1}h)}W^{h,g}\\
&=\frac{1}{|G|}\sum_{g,h\in G}\zt^k(h^{-1}g)W^{h,g}=Q^k\\
\end{align*}
\begin{align*}
\sum_{k\in G}Q^k=&\sum_{k\in G}\frac{1}{|G|}\sum_{g,h\in G}\zt^k(g^{-1}h)W^{g,h}\\
&=\frac{1}{|G|}\sum_{g,h\in G}\dt(g,h)W^{g,h}\\
&=\sum_{g\in G}W^{g,g}=1\\
\end{align*}

\begin{align*}
 \af_g(Q^k)&=\pi_n(g)Q^{k}\pi_n(g)^*\\
&=(\sum_{l\in G}{\zt^{g^{-1}}(l)P_n^l)}(\frac{1}{|G|}\sum_{r,h\in G}{\zt^k(r^{-1}h)W^{r,h}})(\sum_{s\in G}{\zt^{g}(s)P_n^s})\\
&=\frac{1}{|G|}\sum_{l,r,h,s\in G}\zt^{g^{-1}}(l)\zt^k(r^{-1}h)\zt^g(s)P_n^lW^{r,h}P_n^s\\
&=\frac{1}{|G|}\sum_{r,h}\zt^{g^{-1}}(r)\zt^k(r^{-1}h)\zt^g(h)W^{r,h}\\
&=\frac{1}{|G|}\sum_{r,h}\zt^{gk}(r^{-1}h)W^{r,h}=Q^{gk}
\end{align*}

Let's consider the projections $\{1\otimes Q^k\}_{k\in G}$. It's easy to check that they are Rokhlin projections for $F$, therefore $\af$ has the strict Rokhlin property.\newline
 
Conversely, Suppose the action has the strict Rokhlin property. Then the dual action $\hat{\af}$ is strictly approximately representable, and therefore induce trivial action on $K_0(C^{\ast}(G,B,\af))$ (See Proposition 2.4 of \cite{Ph-F} for details). Now fix $n\in \mathbb{N}$. For each $h\in G$, let $I_h=(0,\dots,1_{A_n}^h,0,\dots)$ be an element of $\oplus_{g\in G} A_n^g$. Let $\phi_{n,\infty}$ be the connection map from $\oplus_{g\in G} A_n^g$ to $\varinjlim{\oplus_{g\in G} A_i^g}$ and let $\eta_h=\phi_{n,\infty}(I_h)$. Since $\hat{\af}$ acts trivially on $K_0(C^{\ast}(G,A,\af))$, we have: $\hat{\af}_h([\eta_k])=[\eta_k]$, for any $h,k\in G$. It's easy to see that $\hat{\af}$ commutes with connecting maps so we get:$[\phi_{n,\infty}(\hat{\af}_h(I_k))]=[\phi_{n,\infty}(I_k)]$. For the $K_0$ groups of the inductive limit, we have $\Ker(\phi_{n,\infty})=\cup_{i=n+1}^{\infty}\Ker\phi_{n,i}$. So there exist some $m>n$, such that $[\phi_{n,m}(\hat{\af}_h(I_k))]=[\phi_{n,m}(I_k)]$. Now by Proposition \ref{DualAction}, we have $\hat{\af}_h(I_k)=I_{hk}$.\newline

Let's define ${P}_{n,m}^g=|G|^{-1}\sum_{h\in G}\zt^h(g)\otimes_{i=n}^{m}\pi_{i}(h)$. By Proposition \ref{C-Map}, $P_{n,m}^g$ are mutually orthogonal projections sum up to 1, such that the connecting map $\phi_{n,m}$ is given by:
\begin{equation*}
 (y_{n-1}^h)_{h\in G}\mapsto (y_n^h)_{h\in G}, \quad \text{where}\,\,\, y_n^h=\sum_{g\in G}{y_{n-1}^g\otimes P_{n,m}^{hg^{-1}}}.   
\end{equation*}

 Hence $\phi_{n,m}(I_k)=(1\otimes P_{n+1,m}^{gk^{-1}})_{g\in G}$. Now  $[\phi_{n,m}(\hat{\af}_h(I_k))]=[\phi_{n,m}(I_k)]$ implies:

\begin{equation*}
 [(1\otimes P_{n+1,m}^{g(hk)^{-1}})_{g\in G}]=[(1\otimes P_{n+1,m}^{gk^{-1}})_{g\in G}],\forall h,k\in G
\end{equation*}
Since $K_0(\oplus A_n^g)=\oplus K_0(A_n^g)$ and $A_n^g=A_n$ is an matrix algebra, equality in the $K_0$ group implies that $1\otimes P_{n+1,m}^{g(hk)^{-1}}$ is Murray-Von Nuemann equivalent to $1\otimes P_{n+1,m}^{gk^{-1}}$, for all $g$,$h$,$k$ in $G$. Therefore, for any $n\in \N$, we can find a $m=m(n)$, such that the projections $P_{n+1,m}^{g}$ are mutually \mvnt. Define a telescope inductively by setting $n_1=1$ and $n_{i+1}=m(n_i)$ completes the proof. 
\end{proof}

Note that two projections are equivalent in a UHF algebra \ifo\ they have the same trace. This enables us to reformulate Proposition \ref{SRP1} in terms of the characters of the unitary representations.
\begin{lem}
Let $P_n^g=|G|^{-1}\sum_{h\in G}\zt^h(g)\pi_{n}(h)$ be the projections defined as in Proposition \ref{C-Map}. Let $d_n$ be the dimension of $H_n$ and let $\Tr$ be the unnormalized trace on $\B(H_n)$. Let $\chi_n=\Tr\circ\pi_n$ be the character of $\pi_n$. Then the projections $P_n^g$ are mutually \mvnt \ifo\ $\chi_n(g)=\dt(g,1_{G})d_n$ for any $g\in G$, where $\dt$ is the Kronecker delta.   
\end{lem}
\begin{proof}
If $\chi_n(g)=\dt(g,1_{G})|G|$ for any $g\in G$, then
\begin{equation*}
\Tr(P_n^g)=|G|^{-1}\sum_{h\in G}\zt^h(g)\chi_{n}(h)=|G|^{-1}\zt^{1_|G|}(g)d_n=\frac{d_n}{|G|},\quad \forall g\in G
\end{equation*}
Hence these projections are mutually \mvnt. 
In the opposite direction, as we noted in Remark \ref{C-MapR}, we can see that $\pi_n(h)=|G|\sum_{g\in G}\zt^{h^{-1}}(g)P_n^g$. If the the projections are \mvnt, then we should have $\Tr(P_n^g)=\frac{d_n}{|G|}$, for all $g\in G$. Hence:
\begin{equation*}
\chi_n(h)=|G|\sum_{g\in G}\zt^{h^{-1}}(g)\frac{d_n}{|G|}=\dt(h^{-1},1_{G})d_n=\dt(h,1_{G})d_n.
\end{equation*}
\end{proof}
The above lemma, together with Proposition \ref{SRP1}, proves that the strict Rokhlin property part of Theorem \ref{main}, in the special case where $G$ is a finite cyclic group. \newline

Now let's turn to the tracial Rokhlin property case. We still assume that the group $G$ is finite cyclic and $\af\colon G\mapsto\Aut(A)$ be a product-type action on a UHF algebra $A$. If $\af$ has the tracial Rokhlin property, then the dual action will act trivially on the trace space of the crossed product. (See Proposition 2.5 of \cite{Ph-C}) This suggest us to study the trace space of the crossed product.\newline

The tracial states in $T(C^{\ast}(G,A,\af))$ are in one-to-one correspondence with sequences of compatible traces $\{(\tau_n)_{n=1}^{\infty}|\, \tau_n\in T(\oplus A_n^g), \tau_n=\tau_{n+1}\circ\phi_{n,n+1}\}$. Let $\tau_n^g$ denote the unique tracial state on $A_n^g$. When the C*-algebra is specified, we use $\tau$ to denote $\tau_n^g$, as it will cause no confusion. It's easy to see that $T(\oplus A_n^g)$ can be parametrized by a standard simplex
\begin{equation*}
\Dt^{|G|}=\{(s^g)_{g\in G}\mid\, \sum_{g\in G}s^g=1,\,\,s_g\geq 0,\forall g\in G \},
\end{equation*} 
where $(s^g)_{g\in G}$ is mapped to $\sum_{g\in G}s^g\tau_n^g$.\newline
Let $\tau_{s_n}$ be the tracical state determined by $s_n=(s_n^g)_{g\in G}$. Computation shows that the compatibility condition is equivalent to that 
\begin{equation*}
s_n^g=\sum_{h\in G}{s_{n+1}^h\tau(P_{n+1}^{hg^{-1}})}. 
\end{equation*}
Where $P_n^g$ are the projections defined as in Proposition \ref{C-Map}. Define $T_{n+1}$ to be the $|G|\times|G|$ matrix whose (h,g)-entry is $\tau(P_{n+1}^{hg^{-1}})$. The compatibility condition can be rewritten as $s_n=T_{n+1}s_{n+1}$. This lead us to study the infinite product $\prod_{n}{T_n}$. \newline
Let $P_{n,m}^g$ be the projections defined in Proposition \ref{C-Map}. We have the following identity:
\begin{equation*}
P_{n,n+1}^g=\sum_{l\in G}{P_n^{gl^{-1}}\otimes P_{n+1}^l}
\end{equation*}
 Let $P_n$ be the $|G|\times|G|$ matrix whose (g,h)-entry is $P_n^{gh^{-1}}$, then the above equations can be rewritten in matrix form: $P_{n,n+1}=P_nP_{n+1}$, where the multiplication of the entries is given by tensor product. Since $\tau(a\otimes b)=\tau(a)\tau(b)$, We see that $T_{n,n+1}=T_nT_{n+1}$. So telescoping corresponds to "Adding bracket in the infinite product". As a priori, $\prod_{n}T_n$ may not be convergent. But we have the following:
\begin{lem}\label{T-Matrix}
Let $T_n$ be the matrix $(\tau(P_n^{gh^{-1}}))_{g,h\in G}$, where $P_n^g$ is defined as in Proposition \ref{C-Map}. Then there exists a telescoping, such that for any $m\in \N$, the infinite product $\prod_{n=m}^{\infty}{T_n}$ converges. The conclusion holds no matter $\af$ has the tracial Rokhlin property or not. 
\end{lem}
\begin{proof}
We first observe that the matrices $T_n$ are circular matrices, they can be simultaneously diagonalized: Let $X$ be the unitary matrix $\frac{1}{\sqrt{|G|}}(\zt^h(g))_{h,g\in G}$. Then 
\begin{equation*}
XT_nX^{\ast}=diag(\lambda_n^{g_1},\lambda_n^{g_2},\dots), \quad \text{where}\,\, \lambda_n^g=\sum_{h\in G}{\zt^g(h)\tau(P_n^h)}.
\end{equation*}
 Hence the convergence of the infinite products of matrices is the same thing as the convergence of infinite product of the corresponding eigenvalues: $\prod_{n}{\lambda_n^g}$. Fix a $g$ in $G$. We can estimate that:
\begin{equation*} 
|\lambda_n^g|\leq \sum_{h\in G}{|\zt^g(h)|\tau(P_n^h)}=\tau(\sum_{h\in G}P_n^h)=\tau(I_n)=1.
\end{equation*}
Hence for any $m\in \N$, the partial products $S^l=\prod_{n=m}^{l}{\lambda}_n^g$ is bounded by 1 for any $l>m$. We can therefore select a sub-sequence so that it converges. But choosing a sub-sequence of the partial product corresponds exactly to a telescoping. We conclude that for any $g\in G$ and any $m>0$, the exist a telescoping, such that the infinite product $\prod_{n\geq m}\lambda_n^g$ converges. Since the composition of telescoping is again a telescoping, we can first use a Cantor's diagonal argument, then use induction on $m$, to find a single telescoping, such that for any $g\in G$ and any $m>0$, the infinite products $\prod_{n=m}^{\infty}\lambda_n^g$ converge. 
\end{proof}

Now we are ready to state the necessary and sufficient conditions for tracial Rokhlin property:
\begin{prp}\label{TRP1}
Let $\af\colon G\rightarrow \otimes_{n=1}^{\infty}A_n$ be a product-type action on a UHF-algebra $A$, where $G$ is finite cyclic. Adopt notations in Lemma \ref{T-Matrix}. The action $\af$ has the tracial Rokhlin property if and only if there exists a telescoping, such that for any $m\in \N$, the limit matrix $\prod_{n\geq m}T_n$ exists and has rank 1.
\end{prp}
\begin{proof}
In one direction, suppose $\prod_{n\geq m}T_n$ has rank 1, for any $m\in \N$. Let $F\subset \otimes_n A_n$ be a finite subset. \Wolog\ we assume that $F\subset \otimes_{n=1}^k A_n$ for some $k\in \N$. Let $\ep>0$ be given. Let $m=k+1$.\newline
 For a circular matrix, it has rank 1 if and only if all the entries are equal. Let E be the matrix such that all the entries are equal to $\frac{1}{|G|}$. Then there exists some $l>m$, such that $\| \prod_{n=m}^lT_n-E\|_{\mathrm{max}}<\ep/|G|$. The discussion before Lemma \ref{T-Matrix} shows that $\prod_{n=m}^lT_n=T_{m,l}$, where the (g,h)-th entry of $T_{m,l}$ is $\tau(P_{m,l}^{gh^{-1}})$. Without loss of generality, we assume that $\tau(P_{m,l}^1)$ is smallest among all the entries of $T_{m,l}$, where $1=1_G$ is the identity of $G$. \newline
Now lets do a similar construction as in proposition \eqref{SRP1}. Since $P_{m,l}^{1}$ has smallest trace among the others, we can, for each $g\in G$, find a partial isometry $W^{1,g}$ such that $W^{1,g}W^{g,1}=P_{m,l}^1$ and $W^{g,1}W^{1,g}$ is a sub-projection of $P_{m,l}^g$. Here we adopt the same notation as in Proposition \ref{SRP1}: $W^{g,1}$ is the conjugate of $W^{1,g}$ and $W^{g,h}=W^{g,1}W^{1,h}$.  Let $Q^k=\sum_{g,h\in G}\zt^k(g^{-1}h)W^{g,h}$. By the same computation, we have:
\begin{enumerate}
\item
$\{Q^k\}_{k\in G}$ are mutually orthogonal projections.
\item 
$\af_g(Q^k)=Q^{gk}$.
\item
Let $Q=\sum_{k\in G}{Q^k}$, $Q=\sum_{g\in G}W^{g,g}$.
\end{enumerate}
Now 
\begin{equation*}
\tau(P_{n,l}^1)\geq \frac{1}{|G|}-\|T-E\|_{\mathrm{max}}>\frac{1-\ep}{|G|}.
\end{equation*}
Since $W^{g,g}$ is a subprojection of $P_{n,l}^g$ which is equivalent to $P_{n,l}^{1}$, we have $\tau(Q)=\sum_{g\in G}\tau(W^{g,g})>1-\ep$, or equivalently $\tau(1-Q)<\ep$
Hence $\{1\otimes Q^k\}$ are tracial Rokhlin projections which commutes with $F$. \newline

For the other direction, assume that the action has the tracial Rokhlin property. Then the dual action is tracially approximately representable and hence induce trivial action on the trace space $T(C^{\ast}(G,B,\af))$. Since  $\prod_{n=m}^{\infty}{T_n}$ converges for all $m>0$, we let $T_{m,\infty}$ denote the limits. Fix a vector $s\in \Dt^{|G|}$, let $s_m=T_{m,\infty}s$, for each $m>0$. Since $T_{m,\infty}=T_mT_{m+1,\infty}$, we see that $\{s_m\}$ form a sequence of compatible traces and hence defines a trace on the crossed product $C^{\ast}(G,B,\af)$. Let $e=(1,0,\dots,0)\in \oplus_{g\in G}A_n^g$, and suppose $s_m=(s_m^g)_{g\in G}$. Now the dual action acts trivially on the trace space means $s_m(\hat{\af}_k(e))=s_m(e)$, for each $k\in G$. This implies $s_m^k=s_m^1$, for all k. Hence the image of $\Dt^{|G|}$ under the transformation $T_{m,\infty}$ is a single point. But the simplex $\Dt^{|G|}$ generates the whole domain of $T_{m,\infty}$, we see that $T_{m,\infty}$ has rank 1, for any $m>0$.
\end{proof}

Just as we did for strict Rokhlin property, we can reformulate Proposition \ref{TRP1} in terms of the characters of the unitary representations, using the following lemma:

\begin{lem}
Suppose $\af\colon G\mapsto \Aut(A)$ be a product-type action, where $G$ is finite cyclic and $A$ is UHF. Let $H_n$ and $\pi_n$ be defined as in Definition \ref{ptact}. Set $d_n$ to be the dimension of $H_n$. Let $\chi_n=\Tr\circ\pi_n$ be the character of $\pi_n$, and let $T_n$ be the circular matrix defined as in Lemma \ref{T-Matrix}. Let $\chi\colon G\mapsto \C$ be the characteristic function on $\{1_G\}$. Then for any $l\in\N$, $T_{l,\infty}$ has rank 1 \ifo\ $\prod_{l\geq m}\frac{\chi_n}{d_n}=\chi$.
\end{lem}

\begin{proof}
\Wolog\ we may assume $l=1$. Recall that
 \begin{equation*}
P_{n,m}^g=|G|^{-1}\sum_{h\in G}\zt^h(g)\otimes_{i=n}^{m}\pi_{i}(h).
\end{equation*}
 Since $\Tr=d_i\tau$, we have:
 \begin{equation}\label{trace}
\tau(P_{1,m}^g)=|G|^{-1}\sum_{h\in G}\zt^h(g)\prod_{i=1}^{m}\frac{\chi_{i}(h)}{d_i}
\end{equation}

Since $T_{1,n}$ are circular matrices, $T_{1,\infty}$ has rank 1 \ifo\ $\lim\limits_{n\rightarrow \infty}T_{1,n}$ tends to $E$, where $E$ is the matrix who entries are all equal to $\frac{1}{|G|}$. This is further equivalent to $\lim\limits_{n\rightarrow \infty} \tau(P_{1,n}^g)=\frac{1}{|G|}$, for all $g\in G$. In Equation \eqref{trace}, if we take the limit,  the same computation as in the strict Rokhlin property case shows that  $\lim\limits_{n\rightarrow \infty} \tau(P_{1,n}^g)=\frac{1}{|G|}$ is equivalent to $\lim\limits_{n\rightarrow \infty} \prod_{i=1}^{n}\frac{\chi_i}{d_i}=\chi$.  
\end{proof}

So far we have proved Theorem \ref{main} for finite cyclic groups. For general finite groups, we can actually reduce the problem to the finite cyclic case, based on the following two observations:
\begin{lem}
Let $\af\colon G\mapsto \Aut(A)$ be an action, where $G$ is finite and $A$ is a unital simple C*-algebra. Let $H$ be a subgroup of $G$. If $\af$ has the strict Rokhlin property, then the induced action $\af|_{H}$ also has the strict Rokhlin property. If $\af$ has the tracial Rokhlin property, then $\af|_{H}$ also has the tracial Rokhlin property.
\end{lem}
\begin{proof}
See Lemma 5.6 of \cite{EPW}. The proof of the strict Rokhlin property case is basically the same.
\end{proof}
\begin{lem}
Let $\pi G\mapsto B(H)$ be a finite dimensional representation. Let $H$ be a subgroup of $G$, $\pi$ will induce a representation $\pi|_H$ on $H$. If $\chi$ is
the character of $\pi$ and $\chi_H$ is the character of $\pi|_H$, we have $\chi_H=\chi |_H$.
\end{lem}

\begin{dfn}\label{Model}
(Model action) Let $r=(r_i)_{i=1}^{\infty}$ and $s=(s_i)_{i=1}^{\infty}$ be two infinite sequences of non-negative integers. Let $\pi_i$ be some arbitrary representation on $\mathbb{C}^{s_i}$, and we write $\pi=(\pi_i)_{i=1}^{\infty}$.
Set
\begin{equation*}
 H_i=\underbrace{l^2(G)\oplus l^2(G)\cdots \oplus l^2(G)}_{r_i}\oplus \mathbb{C}^{s_i}.
\end{equation*}
Let $\tilde{\pi_i}\colon G\mapsto B(H_i)$ be the direct sum of left regular representations on each copy of $l^2(G)$ and $\pi_i$ on $\mathbb{C}^{s_i}$. As in definition \ref{ptact}, we get a product-type action $\af(r,s,\pi)$ induced by the representations $\tilde{\pi_i}$. We call $\af(r,s,\pi)$ the {\emph{model action}} for the triple $(r,s,\pi)$. If $s=0$, we write $\af(r)=\af(r,0,0)$.
\end{dfn}
Now let's turn to the proof of Theorem \ref{main} for general finite groups.\newline

Let $\af\colon G\mapsto \Aut(A)$ be a product-type action with the strict Rokhlin property. Write $G$ as a finite union of cyclic subgroups $G=K_0\cup K_1\cup\cdots\cup K_s$. Since $\af|_{K_0}$ has the strict Rokhlin property, by the cyclic group version of Theorem \ref{main}, there exists a telescope such that $\frac{\chi_n|_{K_0}}{d_n}=\chi|_{K_0}$ for any $n\in \N$. We can find a successive telescope such that $\chi_m|_{K_1}=\chi|_{K_1}$ for any $m\in \N$. It's easy to see that we still have $\frac{\chi_m|_{K_0}}{d_m}=\chi|_{K_0}$ after the telescope. Since the composition of telescopes is again a telescope, repeating the above procedure will give us a telescope, such that $\chi_n|_{K_i}=\chi|_{K_i}$ for any $1\leq i\leq s$. Since $G=K_0\cup K_1\cup\cdots\cup K_s$, we see that  $\chi_n=\chi$, for any $n\in \N$.\newline
Conversely, if there exists a telescope, such that $\frac{\chi_n}{d_n}=\chi$ for any $n\in \N$. Let $\io_1,\dots,\io_k$ be the irreducible characters of $G$ with dimensions $r_1,\dots,r_k$ respectively. Let $\chi_n=\sum_{i=1}^{k}a_i\io_i$ be the irreducible decomposition, then $a_i=<\chi_n,\io_i>=\frac{d_nr_i}{|G|}$. Since each $a_i$ is an integer and at lease one $r_i=1$, we see that $\frac{d_n}{|G|}$ is an integer. Hence $\pi_n$ is equal to the character of the direct sum of $\frac{d_n}{|G|}$ copies of left regular representations. Therefore $\pi_n$ is equivalent to a direct sum of left regular representations. Recall that two actions $\af\colon G\mapsto \Aut(A)$ and $\bt\colon G\mapsto \Aut(B)$ are said to be $\emph{conjugate}$ \ifo\ there exists an isomorphism $T\colon A\mapsto B$ such that $T\circ\af_g=\bt_g\circ T$, for any $g\in G$. Let $r=(\frac{d_n}{|G|})_{n>0}$, we see that $\af$ is conjugate to the model action $\af(r)$ (Definition \ref{Model}). Hence $\af$ has the strict Rokhlin property.\newline

The proof of the tracial Rokhlin property case is quite similar, we shall only prove one direction: if there exists a telescoping, such that for any $n\in \N$, $\prod_{i=n}^{\infty}\frac{\chi_n}{d_n}=\chi$, then the action has the tracial Rokhlin property. The following lemma, as a special case of Lemma 5.2 of \cite{EPW}, simplifies our argument for product-type actions.
\begin{prp}\label{TRP-UHF}
Let $\af\colon G\mapsto \Aut(A)$ be a finite group action on a UHF-algebra $A$, let $\ta$ be the unique trace on $A$. Then $\af$ has the tracial Rokhlin property \ifo\ for any finite set $F\subset A$, any $\ep>0$, there exists \mops\ $e_g$ in $A$ for $g\in G$, such that:
\begin{enumerate}[(1)]
\item
$\|e_g a-a e_g\|<\ep$, $\forall g\in G$ and $a\in F$.
\item
$\|\af_g(e_h)-e_{gh}\|<\ep.$
\item
$\tau(1-e)<\ep$
\end{enumerate}
\end{prp}

Now let $F$ be a finite subset of $A$, \wolog\ assume that $F\subset A_{n-1}$, for some $n>0$. Let $\ep>0$ be given. Set $\ep_0=\frac{\ep}{2|G|}$. Since $\prod_{i=n}^{\infty}\frac{\chi_i}{d_i}=\chi$, we can find some $m>n$ such that $\|\prod_{i=n}^{m}\frac{\chi_i}{d_i}-\chi\|_{\max}<\ep_0$. Let $\chi_{n,m}=\prod_{i=n}^{m}\chi_i$, we can see that $\chi_{n,m}$ is the character of the representation $\pi_{n,m}=\otimes_{i=n}^{m}\pi_i$, with dimension $d_{n,m}=\prod_{i=n}^{m}d_i$. Let $M=2|G|/\ep$, increasing $m$ if necessary, we may further require that $d_{n,m}>M$.\newline
In the following, we are going to show that $\pi_{n,m}$ is 'close' to a direct sum of left regular representations. Let $\io_1,\dots,\io_k$ be the irreducible characters of $G$ with dimension $r_1,\dots,r_k$ respectively. Then the max norm of each $\io_i$ will be less or equal to $|G|$. From now on, for characters, $\|\cdot\|$ will always denote the max norm. Let $\chi_{n,m}=\sum_{1\leq i\leq k}a_i\io_i$ be the irreducible decomposition of $\chi_{n,m}$. Since $\|\frac{\chi_{j,k}}{d_{j,k}}-\chi\|<\ep$, We can see that for any $i$,
\begin{align}
 \left|\frac{a_i}{d_{n,m}}-\frac{r_i}{|G|}\right|&=<\frac{\chi_{n,m}}{d_{n,m}}-\chi,\io_i>\\\label{est-d}
&<|G|\left\|\frac{\chi_{n,m}}{d_{n,m}}-\chi\right\|\|\io_i\|<|G|^2\ep_0 
\end{align}
Let $d=\min_{1\leq i\leq k}\{\bigl[\frac{a_i}{r_i}\bigr]\}$. We can then decompose $\chi_{n,m}$ as the sum of two characters $\chi'$ and $\chi''$, where $\chi'=\sum_{i}(dr_i)\io_i$, and $\chi''=\chi_n,m-\chi'$. Let $\pi'$ be the direct sum $d$ copies of left regular representations which corresponds to $\chi'$, and let $\pi''$ be a representation corresponds to $\chi''$. Let $d'$ and $d''$ be the dimensions $\pi'$ and $\pi''$ respectively. Our claim is that $\frac{d''}{d'+d''}=\frac{d''}{d_{n,m}}<\ep$. \\
Note that $d'=\sum_{1\leq i\leq k}dr_i^2=d|G|$. By the definition of $d$, there exists some $i$ such that $\|d-a_i/r_i\|<1$. Using equation (\ref{est-d}), we can estimate:
\begin{align*}
\left|\frac{d''}{d_{n,m}}\right|&=\left\|1-\frac{d|G|}{d_{n,m}}\right\|\\
&\leq \left\|1-\frac{\frac{a_i}{r_i}|G|}{d_n,m}\right\|+\left\|\frac{(d-\frac{a_i}{r_i})|G|}{d_n,m}\right\|\\
&<\frac{r_i}{|G|}|G|^2\ep_0+\frac{|G|}{M}\\
&<\ep/2+\ep/2=\ep.
\end{align*}
Let's consider the representation $\tilde{\pi}=\pi'\oplus\pi''$, Let $\tilde{\af}$ be the inner action defined by $g\mapsto \Ad(\tilde{\pi}(g))$. Since $\tilde{\pi}$ contains copies of regular representation whose total dimension is $d'$, we could find \mops\ $e_g$ for $g\in G$ such that $\tilde{\af}_h(e_g)=e_{hg}$, and $\Tr(e)=d'$, where $e=\sum_{g\in G}e_g$. But $\pi_{n,m}$ is equivalent to $\tilde{\pi}$ because they have the same character, therefore the induced actions $\af_{n,m}$ and $\tilde{\af}$ are conjugate. Hence we can find projections for $\af_{n,m}$ satisfy the same properties. Note that $\Tr(e)=d'$ implies $\tau(1-e)=1-d'/d_{n,m}<\ep$. Hence they are tracial Rokhlin projections. By Lemma \ref{TRP-UHF}, $\af$ has the tracial Rokhlin property.\newline

From the above proof, we can also get the following characterization of the Rokhlin properties:
\begin{cor}
Let $\af\colon G\mapsto \Aut(A)$ be a product-type action where $A$ is UHF. Then:
\begin{enumerate}[(i)]
\item
$\af$ has strict Rokhlin property \ifo\ there exists some $r=(r_i)_{1\leq i<\infty}$, such that $\af$ is conjugate to the model action $\af(r)$ (See Definition \ref{Model}).
\item
$\af$ has the tracial Rokhlin property \ifo\ there exists some $r=(r_i)$, $s=(s_i)$ and $\pi=(\pi_i)$ with $\lim\limits_{i\rightarrow \infty}\frac{s_i}{r_i}=0$, such that $af$ is conjugate to the model action $\af(r,s,\pi)$  
\end{enumerate}
\end{cor}

\section{Tracial Rokhlin property for non-simple C*-algebras}
In this section, we are going to give an alternative definition of tracial Rokhlin property for non-simple C*-algebras. Although the original definition of tracial Rokhlin property makes sense for non-simple C*-algebras, it may be too strong to be distinctive from the strict Rokhlin property, as we can see from the following example:
\begin{exa}\label{wTRP-exa}
Let $\af\colon G\mapsto \Aut(A)$ and $\bt\colon G\mapsto \Aut(B)$ be two actions of the same group $G$. Let $\pi\colon G\mapsto \Aut(A\oplus B)$ be the direct sum, i.e.\, $\pi(a,b)=(\af(a),\bt(b)),$ for any $a\in G$ and $b\in B$. Then $\pi$ has the tracial Rokhlin property, if and only if both $\af$ and $\bt$ have the strict Rokhlin property. In other words, $\pi$ has the tracial Rokhlin property \ifo\ it has the strict Rokhlin property.
\end{exa} 
\begin{proof}
Suppose $\pi$ has the tracial Rokhlin property. Let $F$ be a finite subset of $A$ and let $\ep>0$. Choose any positive element $b$ in $B$ with norm 1. Let $F'=\{(a,0)\mid a\in F\}$ and let $x=(0,b)$. Since $\pi$ has the tracial Rokhlin property, there are \mops\ $e_g$ in $A\oplus B$, for $g\in G$ such that
\begin{enumerate}[(1)]
\item
$\|\af_g(e_h)-e_{gh}\|<\ep$.
\item
$\|e_g(a,0)-(a,0)e_g\|<\ep, \forall a\in F$.
\item
With $e=\sum_{g\in G}e_g$, $1-e\pa x=(0,b)$.
\end{enumerate} 
Let $e_g=(p_g,q_g)$. Then we can see that the projections $p_g\in A$ for $g$ in $G$ are \mops\ satisfy:
\begin{enumerate}[(1')]
\item
$\|\af_g(p_h)-e_{gh}\|\leq \|\af_g(e_h)-e_{gh}\|<\ep$.
\item
$\|p_ga-ap_g\|\leq \|e_g(a,0)-(a,0)e_g\|<\ep, \forall a\in F$.
\item
Let $p=\sum_{g\in G}p_g$, and $q=\sum_{g\in G}q_g$, Then $1-e=(1-p,1-q)\pa(0,b)$, hence $1-p\pa 0$, which forces $1-p=0$, or $p=1$.
\end{enumerate}
Hence $\af$ has strict Rokhlin property. The same argument shows that $\bt$ also has the strict Rokhlin property. It's not hard to see that $\pi=\af\oplus \bt$ has the strict Rokhlin property \ifo\ both $\af$ and $\bt$ has the strict Rokhlin property.
\end{proof}

An element $a$ in a \ca\ is said to be {\emph{full}} if the closed ideal generated by $a$ is the whole \ca. Inspired by the above observation, we give the following alternative definition of tracial Rokhlin property:
\begin{dfn}\label{wTRP}
Let $A$ be an infinite dimensional unital \ca, and let $\af\colon G \rightarrow \Aut(A)$ be an action of a finite group $G$ on $A$. We say that $\af$  has the {\emph{weak tracial Rokhlin property}} if for every finite set $F\subset A$, every $\ep>0$, every positive element $b\in A$ with norm 1 and every {\emph{full}} positive element $x\in A$, there are mutually orthogonal projections $e_g\in A$ for $g\in G$  such that:
\begin{enumerate}[(1)]
\item
$\|\af_g(e_h)-e_{gh}\|<\ep$ for all $g, h\in G$.
\item
$\|e_g a-ae_g\|<\ep$ for all $g\in G$ and all $a\in F$
\item
Let $e=\sum_{g\in G}e_g$, we have $1-e\pa x$ 
\item
$\|ebe\|>1-\ep$
\end{enumerate}
\end{dfn}

By the same perturbation argument as in Lemma 1.17 of \cite{Ph-F}, we could have a formally stronger version of weak tracial Rokhlin property, by requiring that the defect projection be $\af$-invariant:
\begin{lem}\label{wTRP-inv}
Let $A$ be an infinite dimensional unital \ca, and let $\af\colon G \rightarrow \Aut(A)$ be an action of a finite group $G$ on $A$. Then $\af$  has the {\emph{weak tracial Rokhlin property}} \ifo\ for every finite set $F\subset A$, every $\ep>0$, every positive element $b\in A$ with norm 1 and every {\emph{full}} positive element $x\in A$, there are mutually orthogonal projections $e_g\in A$ for $g\in G$  such that:
\begin{enumerate}[(1)]
\item
$\|\af_g(e_h)-e_{gh}\|<\ep$ for all $g, h\in G$.
\item
$\|e_g a-ae_g\|<\ep$ for all $g\in G$ and all $a\in F$.
\item
Let $e=\sum_{g\in G}e_g$, $e$ is $\af-$invariant.
\item
$1-e\pa x$.
\item
$\|ebe\|>1-\ep$.
\end{enumerate}
\end{lem}

The weak tracial Rokhlin property coincides with the original tracial Rokhlin property in the simple \ca\ case:
\begin{prp}
Let $A$ be an infinite dimensional simple unital \ca, and let $\af\colon G\rightarrow \Aut(A)$ be a finite group action. Then $\af$ has the tracial Rokhlin property \ifo\ it has the weak tracial Rokhlin property.
\end{prp}
\begin{proof}
It's trivial that weak tracial Rokhlin property implies tracial Rokhlin property, since every non-zero element in a simple \ca\ is full. So let's prove the other direction. We may assume that $A$ has Property (SP), otherwise $\af$ has the strict Rokhlin property and therefore tracial Rokhlin property. Let $F$ be a finite subset of $A$, $\ep>0$ be a positive number, $b\in A_+$ has norm 1, and $x\in A_+$ is non-zero. Let $\dt=\frac{\ep}{(2|G|+3)}$. By Lemma ~\ref{SP}, there exists a non-zero projection $q\in \Her(b)$, such that for any projection $r\leq q$, we have $\|rb-r\|<\dt$. By Lemma 3.5.6 of ~\cite{Lin-C}, we can find a non-zero projection $p\leq q$, such that $p\pa x$. Let $F'=F\cup \{p\}$, since $\af$ has tracial Rokhlin property, we can \mops\ $e_g\in A$ such that:
\begin{enumerate}[(1)]
\item
$\|\af_g(e_h)-e_{gh}\|<\dt<\ep$ for all $g, h\in G$.
\item
$\|e_g a-ae_g\|<\dt<\ep$ for all $g\in G$ and all $a\in F'$
\item
Let $e=\sum_{g\in G}e_g$, we have $1-e\pa p$ 
\item
$\|epe\|>1-\dt$
\end{enumerate}
By our choice of $p$, we have: $1-e\pa p\pa x$. So we need only to verify that $\|ebe\|>1-\ep$. For that, we have the following estimation:
\begin{align*}
\|ebe\|&\geq \|pebep\|>\|epbpe\|-\|(ep-pe)bpe\|-\|peb(pe-ep)\|\\
&>\|epe\|-\|e(p-pb^{1/2})pe\|-\|epb^{1/2}(b^{1/2}p-p)e\|-2|G|\dt\\
&>1-\dt-\dt-\dt-2|G|\dt=1-\ep
\end{align*}  
\end{proof}
Looking back to Example \ref{wTRP-exa} we gave at the beginning of this section, we see that the weak tracial Rokhlin property is a better definition, because we have the following:
\begin{prp}
Let $\af\colon G\mapsto \Aut(A)$ and $\bt\colon G\mapsto \Aut(B)$ be two actions of the same group $G$. Let $\pi\colon G\mapsto \Aut(A\oplus B)$ be the direct sum, i.e.\, $\pi(a,b)=(\af(a),\bt(b)),$ for any $a\in G$ and $b\in B$. Then $\pi$ has the weak tracial Rokhlin property, if and only if both $\af$ and $\bt$ have the weak tracial Rokhlin property.
\end{prp} 
For the proof, we just need to observe that for $a=(a_1, a_2)$ in $A\oplus B$, $a$ is full \ifo\ $a_1$ and $a_2$ are both full in the corresponding C*-algebras.\newline

The original motivation for introducing the tracial Rokhlin property is to deal with crossed products of tracially AF algebras, or C*-algebras of tracial rank 0. In \cite{Lin-TR1} Lin gives the following definition:   
\begin{dfn}
We denote by $\mathscr{I}^{(k)}$ be the class of all unital C*-algebras which are finite direct sum of the form:
\begin{equation*}
 P_1M_{n_1}(C(X_1))P_1\oplus P_1M_{n_2}(C(X_2))P_2\oplus\cdots\oplus P_sM_{n_s}(C(X_s))P_s  
\end{equation*}
where $s<\infty$ and for each $i$, $X_i$ is a $k$-dimensional finite CW complex, $P_i$ is a projection in $M_{n_i}(C(X_i))$
\end{dfn}
These will serve as our building blocks. Lin gave a definition of tracial rank in \cite{Lin-TR1}, where it was called tracial topological rank. Later he showed in \cite{Lin-TR2} that the definition can be slightly simplified, for which we can take as the definition of tracial rank: 
\begin{thm}\label{TR}
(Theorem 2.5,\cite{Lin-TR2}) Let $A$ be a unital \ca. Then the tracial rank of $A$ is less or equal to k, \ifo\ for any finite subset $F\in A$, every $0<\sm_4<\sm_3<\sm_2<\sm_1<1,$ and every positive element $b\in A$ with $\| b\|=1,$, there exist a C*-subalgebra $B\subset A$ with $B\in \mathscr{I}^{(k)}$ and $1_B=p$ such that:
\begin{enumerate}[(1)]
\item
$\|[x,p]\|<\ep$ for all $x\in F$.
\item
$pxp\in_{\ep}B$ for all $x\in F$.
\item
$f^{\sm_1}_{\sm_2}( (1-p)b(1-p) ) \pa f^{\sm_3}_{\sm_4}( (pbp) )$
\end{enumerate}
\end{thm}

If the tracial rank of $A$ is less or equal to $k$, we write $TR(A)\leq k$. While the comparison condition (3) in the definition of tracial rank seems rather complicated, it can be greatly simplified when the C*-algebra is simple:

\begin{thm}\label{TR-S}
(Theorem 6.13,\cite{Lin-TR1}) Let $A$ be a simple unital \ca. Then $TR(A)\leq k$ if for any finite subset $F\in A$, any nonzero element $b\in A_+$, there exist a C*-subalgebra $B\subset A$ with $B\in \mathscr{I}^{(k)}$ and $1_B=p$ such that:
\begin{enumerate}[(1)]
\item
$\|[x,p]\|<\ep$ for all $x\in F$.
\item
$pxp\in_{\ep}B$ for all $x\in F$.
\item
$1-p\pa b$.
\end{enumerate}
\end{thm}

In \cite{Ph-F}, Phillips proved the following:
\begin{thm}\label{Ph-TR0}
Let $A$ be an infinite dimensional simple separable unital C*-algebra with tracial rank zero. Let $\af\colon G\rightarrow \Aut(A)$ be an action of a finite group $G$ on $A$ with the tracial Rokhlin property. Then the crossed product $C^*(G,A,\af)$ has tracial rank zero.
\end{thm}

One natural question would be how to extend the above result to the non-simple case. The main result of this section is that our definition of weak tracial Rokhlin property works in some special case, namely when the crossed product is simple: 

\begin{thm}\label{Main}
Let $A$ be an infinite dimensional separable unital \ca\. Let $\af\colon G\rightarrow Aut(A)$ be an action with the weak tracial Rokhlin property. Assume that $A$ is $\af$ simple and $TR(A)\leq k$. Then $C^*(A, G, \af)$ has tracial rank $\leq k$.\newline
\end{thm}

We shall prove this theorem at the end of this section.\newline

The assumption that $A$ is $\af$ simple is used to ensure that $C^*(G,A,\af)$ is simple, as we have the following:
\begin{lem}\label{simple}
Let $A$ be an unital C*-algebra. Let $\af\colon G\rightarrow \Aut(A)$ be a finite group action with the weak tracial Rokhlin property. Then $A$ is $\af$-simple \ifo\ $C^*(G,A,\af)$ is simple.
\end{lem}
\begin{proof}
If $I$ is a proper $\af$-invariant ideal in $A$, then $C^*(G, I,\af)$ is a proper ideal in $C^*(G,A,\af)$, hence $C^*(G,A,\af)$ is simple implies $A$ is $\af$-simple. For the other direction, there are several proofs. Using the Rokhlin projections, it's not hard to show that for any $a_0,a_1,\dots, a_n$ in $A$ with $a_0$ positive, any $\ep>0$, there is a projection $p$ such that:
\begin{equation*}
\|pa_0p\|>\|a_0\|-\ep, \quad \|pa_i-a_ip\|<\ep, \quad \|p\af_g(p)\|<\ep \;\text{for}\; g\neq 1 \quad (*)
\end{equation*}
 Then the same lines as Theorem 7.2 of \cite{OPed} shows that $C^*(G,A,\af)$ is simple. \newline
A more sophisticated way is to use Theorem 2.5 of \cite{S-IS}. A discrete group is exact if and only if the reduced C*-algebra is exact. In particular, finite groups are exact. When $A$ is $\af-$simple, it's easy to see that the weak tracial Rokhlin property implies the residual Rokhlin$^*$ property (See Definition 2.1, \cite{S-IS}). By Theorem 2.5 of \cite{S-IS}, the $A$ separates the ideal in the crossed product, which is equivalent to that the crossed product is simple. 
\end{proof}

When $A$ is $\af-$simple, we have a similar result as in Lemma \ref{SP-SRP}. It may not be true in general.
\begin{lem}\label{SP-SRP2}
Let $A$ be an unital C*-algebra. Let $\af\colon G\rightarrow \Aut(A)$ be a finite group action with the weak tracial Rokhlin property. If $A$ is $\af$-simple, then either $A$ has Property (SP) or $\af$ has the strict Rokhlin property
\end{lem}
\begin{proof}
When $A$ is $\af$-simple, by Proposition 2.1 of \cite{R-F}, $A$ is a finite direct sum of simple ideals which are permuted transitively among each other by the action of $G$. Write $A=I_1\oplus\cdots\oplus I_n$. Suppose $A$ does not have Property (SP). Then there exists some non-zero positive element $b\in A$ such that $\Her(b)$ contains no non-zero projection. Write $b=(b_1,\dots,b_n)$, where $b_i$ is a positive element in $I_i$. Since $b$ is non-zero, \Wolog\ $b_1$ is non-zero. Hence $\Her(b_1)$ contains no non-zero projection. Since the ideals are permuted transitively by the action of $G$, we can choose $g_i$ such that $\af_{g_i}(I_1)=I_i$, for each $i$. Let $b_i'=\af_{g_i}(b_1)$, consider the element $b'=(b_1',\dots, b_n')$. Since $b_1'$ is a full element in $I_1$, we can see that $b'$ is a full element in $A$. But $\Her(b')=\Her(b_1')\oplus\cdots \oplus\Her(b_n')$ contains no non-zero projections. Now if we choose Rokhlin projections correspond to $b'$, the defect projection will be forced to be 0, hence the Rokhlin projections must sum up to 1 and therefore $\af$ has the strict Rokhlin property. 
\end{proof}

It should be pointed out that to the author's knowledge, so far no versions of Rokhlin properties are known that could extend Theorem \ref{Ph-TR0} in full generality, except the strict Rokhlin property:

\begin{thm}\label{SRP-TR}
Let $A$ be an unital separable \ca\ with tracial rank $\leq k$. Let $\af\colon G\rightarrow Aut(A)$ be an action with the strict Rokhlin property. Then $C^*(A, G, \af)$ has tracial rank $\leq k$
\end{thm}

The proof of the above theorem is based on the following lemma, which can be extracted from the proof of Theorem 2.2 in \cite{Ph-F}:
\begin{prp}\label{SRP-C}
Let $A$ be a unital \ca\, $\af\colon G\rightarrow \Aut(A)$ be an action with the strict Rokhlin property. Let $n=\card(G)<\infty$. Then for any finite set $F$in the crossed product $C^*(G,A,\af)$, any $\ep>0$, there exists a projection $f\in A$ and a unital homomorphism $\phi\colon M_n(fAf)\rightarrow C^*(G,A,\af)$ such that $\dist(b,\phi(M_n(fAf))<\ep$, for any $b\in F$.
\end{prp}

$\mathbb{Proof of Theorem \ref{SRP-TR}}\colon$ The property of having traical rank less or equal to $k$ passes to unital hereditary subalgebra, matrix algebra and sub-quotients, by Proposition 5.1, Theorem 5.3 and Theorem 5.8 of \cite{Lin-TR1}. Hence Proposition \ref{SRP-C} together with Proposition 4.8, proves our Theorem \ref{SRP-TR}.\newline

In Theorem \ref{Main}, the assumption of Property (SP) on $A$ is used to ensure that the crossed product has Property (SP). The proof of this fact depends on the following theorem:

\begin{thm}\label{O1}
(Theorem 2.1 of \cite{Os1}) Let $1\in A\subset B$ be a pair of C*-algebras. Suppose that $A$ has Property (SP). If there is a faithful conditional expectation $\E$ from $B$ to $A$, such that for any non-zero positive element $x$ in $B$ and any $\ep>0$, there is an element $y$ in $A$ satisfying: 
\begin{equation*}
\|y^*(x-\E(x))y\|<\ep, \quad \|y^*\E(x)y\|\geq \|\E(x)\|-\ep,
\end{equation*}
then $B$ also has Property (SP). Moreover, every non-zero hereditary \ca\ of $B$ has a projection which is equivalent to some projection in $A$ in the sense of Murray-von Neumann.
\end{thm}

\begin{prp}\label{SP-cross}
Let $A$ be a \ca\ with Property (SP) and let $\af\colon G\rightarrow Aut(A)$ be an action with the weak tracial Rokhlin property. Then $C^*(A, G, \af)$ also has Porperty (SP). Moreover, every non-zero hereditary \ca\ of  $C^*(A, G, \af)$ has a projection which is equivalent to some projection in $A$ in the sense of Murray-von Neumann.
\end{prp}

\begin{proof} 
Let $B$ be the crossed product $C^*(A, G, \af)$. Let $\{u_g\}$ be the canonical unitaries. Then every element of $B$ can be written as $\sum_{g\in G}b_gu_g$, where $b_g\in A, \forall g\in G$. Let $1$ be the identity element of $G$, then a natural faithful expectation $\E$ from $B$ to $A$ can be defined as $\E(\sum_{g\in G}b_gu_g)=b_1$. 

Now we check that $\E$ satisfy the conditions in Theorem \ref{O1}. Let $x=\sum_{g\in G}b_gu_g$ be a non-zero positive element of $B$. Then $b_1=\E(x)$ must be a non-zero positive element of $A$: write $x=zz^*$ where $z=\sum_{g\in G}c_gu_g$, then $b_1=\sum_{g\in G}c_gc_g^*\neq 0$. Without loss of generality, we can assume that $b_1$ has norm 1.  Since $\af$ has the weak tracial Rokhlin property, Let $F=\{b_g\mid g\in G\}$, $b_1$ be the positive element of $A$ with norm 1 and $\dt=\frac{\ep}{(|G|^2-|G|+1)}$, we can find mutually orthogonal projections $e_g\in A$ for $g\in G$  such that:
\begin{enumerate}[(1)]
\item
$\|\af_g(e_h)-e_{gh}\|<\dt$ for all $g, h\in G$.
\item
$\|e_g a-ae_g\|<\dt$ for all $g\in G$ and all $a\in F$. 
\item
With $e=\sum_{g\in G}e_g$, $\|eb_1e\|>1-\dt$.
\end{enumerate}
Since $\|eb_1-b_1e\|<\dt$, we see that:
\begin{equation*}
\|eb_1e-\sum_{g\in G}e_gb_1e_g)\|=\|\sum_{g\neq h}e_gb_1e_h\|\leq (|G|^2-|G|)\dt+\|\sum_{g\neq h}e_ge_hb_1\|=(|G|^2-|G|)\dt. 
\end{equation*}
Note that since $e_gb_1e_g$ for $g\in G$ are orthogonal to each other, we have 
\begin{equation*}
\|\sum_{g\in G}e_gb_1e_g\|=\max\{\|e_gb_1e_g\|\mid g\in G\}.
\end{equation*}
Hence there exists some $h\in G$, such that 
\begin{equation*}
\|e_hb_1e_h\|=\|\sum_{g\in G}e_gb_1e_g\|>\|eb_1e\|-(|G|^2-|G|)dt>1-\ep .
\end{equation*}
Let $y=e_h$, then $\|y\E(x)y\|>1-\ep$. Also we have:
\begin{align*}
\|y*(x-\E(x))y\|\\
&=\|e_h(\sum_{g\neq 1}b_gu_g)e_h\|\\
&=\|(\sum_{g\neq 1}e_hb_g\af_g(e_h)u_g\|\\
&\leq  (\sum_{g\neq 1}\|e_hb_ge_{gh}\|\\
&\leq (\sum_{g\neq 1}(\|(e_hb_g-b_ge_h)e_{gh}\|+\|b_ge_he_{gh}\|)\\
&\leq (|G|-1)(\dt+0)<\ep
\end{align*} 
Hence the conditions in Theorem ~\ref{O1} are met.Note that the comparison condition is not used in the above proof.
\end{proof}

We can extract the following technical lemma from the proof of Theorem 2.2, \cite{Ph-F}:
\begin{lem}\label{matrix}
Let $\af\colon G\rightarrow \Aut(A)$ be a finite group action. Let $F$ be a finite subset of $A$ and let $u_g$, $g\in G$ be the canonical unitaries implementing the action. Set $n=\card(G)$. For any $\ep>0$, there exist $\dt>0$ such that for any family of \mops\ $e_g\in A$, for $g\in G$ with:
\begin{enumerate}[(1)]
\item
$\|\af_g(e_h)-e_{gh}\|<\dt$,
\item
$\|e_ga-ae_g\|<\dt$ for any $a\in F$ and 
\item
$e=\sum_{g\in G}e_g$ is $\af-$invariant, 
\end{enumerate}

then there exists a unital homomorphism $\phi_0\colon M_n\rightarrow C^*(G,A,\af)$ such that $\phi_0(v_{g,g})=e_g$, where $v_{g,h}$ is the standard (g,h)- matrix units of $M_n$. Furthermore, let $1=1_G$ be the identity of $G$, if we define an unital homomorphism $\phi\colon M_n\otimes e_1Ae_1\rightarrow eC^*(G,A,\af)e$ by:
\begin{equation*}
 \phi(v_{g,h}\otimes a)=\phi_0(v_{g,1}a\phi_0(v_{1, h}),\quad \text{for} g, h\in G \quad text{and} \quad a\in e_1Ae_1.  
\end{equation*}
There is a finite subset $T$ of $M_n\otimes e_1Ae_1$ such that for every $a\in F\cup\{u_g\mid g\in G\}$, there is some $b\in T$ such that $\|\phi(b)-eae\|<\ep$. 
\end{lem}

$\mathbb{Proof of Theorem \ref{Main}\colon}$ The proof is actually a modification of Theorem 2.6 of \cite{Ph-F}. Let $B=C^*(A, G, \af)$. Since the action is $\af$-simple, $B$ is simple by Lemma \ref{simple}. Note that for simple C*-algebras, tracial topological rank can be simplified as in Theorem \ref{TR-S}. By Lemma \ref{SP-SRP2} and Theorem \ref{SRP-TR}, we can assume that $A$ has Property (SP).\newline

Let $S$ be a finite subset of $B$. \Wolog\ we may assume that $S$ is of the form $F\cup \{u_g\colon g\in G\}$, where $F$ is a finite subset of the unit ball of $A$ and $u_g\in B$ are the canonical unitaries implementing the automorphism $\af_g$. Let $\ep>0$ be a positive number, and let $x$ be a nonzero positive element of $B$.\\
Set $n=|G|$, By Lemma \ref{SP-cross}, $B$ has Property (SP). Also $B$ is non-elementary since it's infinite dimensional unital simple. By Lemma \ref{SP2}, we can find $2n$ non-zero mutually orthogonal and equivalent projections $p_1, p_2, \cdots, p_{2n}$ in $\Her(x)$. Then by Lemma \ref{SP-cross} again, we can find a projection $p'\in A$ such that $p'\pa p_1$.\newline

Now we are going to find a non-zero sub-projection $p$ of $p'$ so that $\af_g(p)\pa p_1$ in $B$, for every $g\in G$. List the elements of $G$ as $g_1,g_2,\cdots,g_n$. Let $f_0\leq p_1$ such that $p'\sim f_0$.  Since $B$ is simple, by Lemma 3.5.6 of \cite{Lin-C}, there exists non-zero projections $f_1'\leq \af_{g_1}(p')$ and $f_1\leq f_0\leq p_1$ such that $f_1'\sim f_2$. Inductively, for $i=1,2,\dots,n$, we can find non-zero projections $f_i$ and $f_i'$, such that：
\begin{equation*}
f_i'\leq \af_{g_ig_{i-1}^{-1}}(f_{i-1}')\leq \af_{g_i}(p'),\quad f_i\leq f_{i-1} \quad\text{and}\quad f_i'\sim f_i.
\end{equation*}
Let $p=\af_{g_n^{-1}}f_n'$. Then $p$ is a non-zero subprojection of $p'$ such that $\af_{g_i}(p)\pa p_1$, for any $i$. \newline 

Since $B$ is simple, there exists $s_1,s_2,\dots, s_n$ in $B$ such that $\sum_{i=1}^ns_ips_i^*=1$. Write $s_i=\sum_{g\in G}s_{i,g}u_g$. Do a computation we can show that 
\begin{equation*}
\sum_{i=1}^n\sum_{g\in G}s_{i,g}\af_g(p)s_{i,g}^*=1
\end{equation*}
Set $\hat{p}=\sum_{g\in G}\af_g(p)$, then 
\begin{equation*}
\sum_{g,i}s_{i,g}\hat{p} s_{i,g}^*>\sum_{g,i} s_{i,g}\af_g(p)s_{i,g}^*=1
\end{equation*}
Hence $\hat{p}$ is full in $A$.  Therefore there exists $\{z_i | 1\leq i \leq m\}\subset A$ such that $\sum_{i=1}^{m}z_ipz_i^*=1$. Let $M=\max\{\|z_i\|\mid 1\leq i \leq m\}$. Set $\ep_0=\ep/4$, then choose $\dt_0>0$ according to Lemma \ref{matrix} for $n$ as given and for $\ep_0$ in place of $\ep$. Let
\begin{equation*}
\dt=\min \{\frac{1}{2nmM},\; \dt_0,\; \frac{\ep}{8n}\}. 
\end{equation*}
Let $F'=F\cup \{z_i | 1\leq i \leq m\}$ be a finite subset of $A$, $\hat{p}$ be a full positive element of $A$ and $\dt>0$. By Lemma \ref{wTRP-inv}, we can obtain \mops\ $e_g$ in $A$ for $g\in G$, such that:
\begin{enumerate}[(1)]
\item
$\|\af_g(e_h)-e_{gh}\|<\dt<\dt_0$ for all $g, h\in G$.
\item
$\|e_g a-ae_g\|<\dt<\dt_0$ for all $g\in G$ and all $a\in F'$
\item
Let $e=\sum_{g\in G}e_g$, $e$ is $\af-$invariant.
\item
$1-e\pa \hat{p}$ 
\end{enumerate}
Let $v_{g,h}$ be the standard matrix units for $M_n$. By the choice of $\dt$, there exists a unital homomorphism $\phi_0\colon M_n\rightarrow eBe$ such that $\phi_0(v_{g,g})=e_g$ for all $g\in G$. Furthermore,if we define a unital homomorphism $\phi\colon M_n\otimes e_1Ae_1\rightarrow eBe$ by $\phi(v_{g,h}\otimes a)=\phi_0(v_{g,1}a\phi_0(v_{1, h})$ for $g, h\in G$ and $a\in e_1Ae_1$, then there is a finite subset $T$ of $M_n\otimes e_1Ae_1$ such that for every $a\in F'\cup{u_g\mid g\in G}$, there is $b\in T$ such that $\|\phi(b)-eae\|\leq \ep_0$.\newline

Now $e_1Ae_1$ is an hereditary C*-subalgebra of $A$ and $TR(A)\leq k$, we have $TR(M_n\otimes e_1Ae_1)\leq k$, by Theorem 5.3 and Theorem 5.8 of \cite{Lin-TR1}. In particular, $TR_w(M_n\otimes e_{g_0}Ae_{g_0})\leq k$. ($TR_w(\cdot)$ is the tracial weak rank, see Definition 3.4 of \cite{Lin-TR1} for the definition; Corollary 5.7 of \cite{Lin-TR1} says that $TR(A)\leq k$ implies $TR_w(A)\leq k$ )   \newline
Now consider the element $r=e_{11}\otimes e_{1}\hat{p}e_{1}$. Since $\sum_{i=1}^nz_i\hat{p}z_i^*=1$, using the fact $\|e_1z_i-z_ie_1\|<\dt$, we see that:
\begin{align*}
&\|e_1-\sum_{i=1}^me_1z_ie_1\hat{p}e_1z_i^*e_1\|\\
=&\|\sum_{i=1}^me_1z_i\hat{p}z_i^*e_1-\sum_{i=1}^me_1z_ie_1\hat{p}e_1z_i^*e_1\|\\
\leq& \|\sum_i(e_1(e_1z_i-z_ie_1)\hat{p})z_i^*e_1\|+\|\sum_i(e_1z_ie_1)\hat{p})(z_i^*e_1-e_1z_i^*)e_1\|\\
\leq& mnM\dt+mnM\dt<1
\end{align*}
This shows that the ideal generated by $e_1\hat{p}e_1$ in $e_1Ae_1$ contains an invertible element, hence $e_{1}\hat{p}e_{1}$ is full in $e_{1}Ae_{1}$. Therefore $r$ is a full element of $M_n\otimes e_{1}Ae_{1}$. By the definition of tracial weak rank,  there is a projection $p_0\in M_n\otimes e_1Ae_1$ and an $E_0\in \mathscr{I}^{(k)}$ with $1_{E_0}=p_0$ such that
\begin{enumerate}
\item
$\|p_0b-bp_0\|<\ep_0$ for all $b\in T.$
\item
For every element $b\in T$, there is an element $b'\in E_0$ such that $\|p_0bp_0-b'\|<\ep_0$
\item
$1-p_0\pa r $.
\end{enumerate}

Set $q=\phi(p_0)$ and $E=\phi(E_0)$. Note that the identity of $E$ is equal to $e$, the sum of the Rokhlin projections. Since $E$ is isomorphic to a sub-quotient of $E_0$ and $E_0\in \mathscr{I}^k$, by the same argument as in Proposition 5.1 of ~\cite{Lin-TR1}, there exists an increasing sequence of C*-algebras $C_i$, such that the union $\cup_{i=1}^{\infty}C_i$ is dense in $E$. Therefore, we can choose some large $i=N$, such that $1_{C_N}=1_E$ and for every $b\in T$, there is an element $b'\in E_0$ and a $b''$ in $C_N$ so that $\|p_0bb_0-b'\|<\ep_0$ and $\|\phi(b')-b''\|<\ep_0$.\newline
 
Let  $a\in S$. Choose $b\in T$ such that $\|\phi(b)-eae\|<\ep_0$. Then, using $qe=eq=q$,
\begin{align*}
\|qa-aq\| &\leq 2\|ea-ae\|+\|qeae-eaeq\|\\
&\leq 2\|ea-ae\|+2\|eae-\phi(b)\|+\|p_0b-bp_0\|\\
&<2n\dt+2\ep_0+\ep_0\leq\ep
\end{align*}

Further, choosing $b'\in E_0$ such that $\|p_0bp_0-b'\|<\ep_0$, and then choose $b''$ in $C_N$ such that $\|\phi(b')-b''\|<\ep_0$. then the element $b''\in C_N$ satisfies
\begin{align*}
\|qaq-b''\|&\leq \|qaq-q\phi(b)q\|+\|q\phi(b)q-\phi(b')\|+\|\phi(b')-b''\|\\
&\leq \|eae-\phi(b)\|+\|\phi(p_0bp_0-b')\|+\|\phi(b')-b''\|\\
&\leq \ep_0+\ep_0+\ep_0\leq \ep.\\
\end{align*}

Finally, for the comparison condition, since $\phi(r)=e_1\hat{p}e_1$:
\begin{align*}
1-q&=(1-e)+(e-q)\pa hat{p}\oplus\phi{1-p_0}\\
&\pa hat{p}\oplus e_1\hat{p}e_1\\
&\pa\oplus_{g\in G}\af_g(p)\oplus_{g\in G}\af_g(p)\\
&\pa p_1\oplus p_2\oplus\cdots\oplus p_2n\pa x.
\end{align*}
Hence $B=C^*(G,A,\af)$ has tracial rank less or equal to $k$, by Theorem \ref{TR}.

\begin{rmk}
Actually we could replace tracial rank by weak tracial rank in Theorem \ref{Main} if $TR_w(eAe)\leq TR_w(A)$, for any projection $e\in A$. But unfortunately this is not true in general. See Example 4.7 of \cite{Lin-TR1} for a counterexample. We can also see that the norm condition is not officially used in the proof of Theorem \ref{Main}. It is used in Proposition \ref{SP-cross} which is an essential ingredient of the proof. It's possible to find a weak condition so that Proposition \ref{SP-cross} still holds, in which case Theorem \ref{Main} is still valid.
\end{rmk}
\section{Weak tracial Rokhlin property and tracial approximation}
In this section, we assume that all classes of C*-algebras that we consider has the property that, if $A\cong B$ and  $A$ belongs to the class, then so is $B$. Following the spirit of tracially AF algebras, Elliott and Niu made the following definition of tracial approximation:
\begin{dfn}\label{TA}
(Definition 2.2, \cite{EN-TA}) Let $\mathscr{I}$ be a class of C*-algebras. A unital C*-algebra $A$ is said to be in the class $\TA\mathscr{I}$, \ifo\ for any $\ep>0$, any finite subset $F$ of $A$, and any nonzero $a\in A^+$, there exist a non-zero projection $p\in A$ and a sub C*-algebra $C\subset A$ such that $C\in \mathscr{I}$, $1_C=p$, and for all $x\in F$, 
\begin{enumerate}
\item
$\|xp-px\|<\ep$,
\item
$pxp\in_{\ep} C$, and
\item
$1-p\pa a$
\end{enumerate}  
\end{dfn}

By the same argument as in Example \ref{wTRP-exa}, we can see that if $A, B$ are two C*-algebras such that $A\oplus B\in \TA\mathscr{I}$, then both $A$ and $B$ are in $\mathscr{I}$. Hence the comparison condition in Definition \ref{TA} may be too strong for non-simple C*-algebras. We could make the following alternative definition:
\begin{dfn}
(Weak tracial approximation) Let $\mathscr{I}$ be a class of C*-algebras. Then we define $w\TA\mathscr{I}$ to be the class of C*-algebra $A$ obtained the same way as in Definition \ref{TA} with the additional requirement that the positive element $a$ be full. 
\end{dfn}
The class $w\TA\mathscr{I}$ properly contains $\TA\mathscr{I}$, and it contains more C*-algebras of interest. But in contrast with $\TA\mathscr{I}$, even if the class $\mathscr{I}$ is closed under taking hereditary sub-algebra, $w\TA\mathscr{I}$ may not have the same property. Hence we need make this assumption in the following proposition:
\begin{thm}\label{wTRP-Main}
Let $A$ be an infinite dimensional unital C*-algebra with Property (SP). Let $G$ be a finite group. Let $\af\colon G\rightarrow \Aut(A)$ be an action with the weak tracial Rokhlin property such that the crossed product $C^*(G,A,\af)$ is simple. Suppose $A$ belongs to a sub-class $\mathscr{I}'$ of $w\TA\mathscr{I}$, for some class of C*-algebras $\mathscr{I}$. If $\mathscr{I}'$ is closed under taking hereditary sub-algebras and tensoring with matrix algebras, then $C^*(G, A, \af)$ belongs to $\TA\mathscr{I}$. In particular, if $A$ belongs to $\TA\mathscr{I}$, and $\mathscr{I}$ is closed under taking hereditary subalgebras and tensoring with matrix algebras, then $C^*(G,A,\af)$ belongs to $\TA\mathscr{I}$.
\end{thm} 
\begin{proof}
Let $F$ be a finite subset of $C^*(G,A,\af)$, let $a$ be a non-zero element of $C^*(G,A,\af)^+$, and let $\ep>0$. Let $e_g$, $g\in G$ be the Rokhlin projections corresponds to $F$, $a$ and $\ep$. Let $e=\sum_{g\in G}e_g$. From the proof of Theorem \ref{Main}, we can find a unital homomorphism $\phi: M_n\otimes (e_1Ae_1) \rightarrow eC^*(G,A,\af)e$, and a subalgebra $C$ of $M_n\otimes (e_1Ae_1)$ with $1_C=p_0$ which is in the class $\mathscr{I}$, such that:
\begin{enumerate}
\item
$\|x\phi(p_0)-\phi(p_0)x\|<\ep$, for every $x\in F$.
\item
$\phi(p_0)x\phi(p_0)\in_{\ep}\phi(C)$, for every $x\in F$.
\item
$1-\phi(p_0)\pa a$
\end{enumerate}  
It's not hard to see that the homomorphism $\phi_0$ defined in the proof of Theorem \ref{Main} is actually injective if $\dt$ is sufficiently small. Since $\mathscr{I}$ contains isomorphic copies of its member, we see that $C^*(G,A,\af)$ belongs to $\TA\mathscr{I}$. By Lemma 2.3 of \cite{EN-TA}, if $\mathscr{I}$ is closed under taking unital hereditary sub-algebras and tensoring with matrix algebras, so is $\TA\mathscr{I}$, hence the theorem follows.
\end{proof}

As a corollary, we have the following:
\begin{cor}
Let $A$ be an infinite dimensional unital separable C*-algebra. Let $\af\colon G\rightarrow \Aut(A)$ be a finite group action with the weak tracial Rokhlin property such that the crossed product $C^*(G,A,\af)$ is simple. Suppose $A$ has stable rank one, then $C^*(G,A,\af)$ also has stable rank one.
\end{cor}
\begin{proof}
First of all, the class of stable rank one C*-algebras is preserved by strict Rokhlin actions. Hence by \ref{SP-SRP2}, we may assume that $A$ has Property (SP). By Theorem 3.18 and Theorem 3.19 of \cite{Lin-C}, the class of unital C*-algebras with stable rank one is closed under taking unital hereditary sub-algebras and tensoring with matrix algebras. It follows from Theorem \ref{wTRP-Main} that the crossed product $C^*(G,A,\af)$ is tracially of stable rank 1. By our assumption, $C^*(G,A,\af)$ is simple. By Theorem 4.3 of \cite{EN-TA}, $C^*(G,A,\af)$ actually has stable rank one.
\end{proof}

For real rank, with some modification of Theorem 4.3, we could get the following:
\begin{lem}\label{RR0}
Let $\mathscr{I}$ be the class of unital C*-algebras with real rank 0. Let $A$ be a simple C*-algebra in $\TA\mathscr{I}$, then $A$ has real rank 0.
\end{lem}
\begin{proof}
Any finite dimensional C*-algebra has real rank 0, hence we may assume that $A$ is infinite dimensional. We may also assume that $A$ has Property (SP). Otherwise $A$ will be locally have real rank 0 and therefore $A$ itself has real rank 0. Let $x$ be a non-zero self-adjoint element in $A$ and let $\ep>0$. Assume that $x$ is not invertible, otherwise there's nothing to prove. We can find a some $\sm>0$ such that $\|f_{\sm}^1(x)-x\|<\ep/2$. We write $f=f_{\sm}^1$. Since $x$ is not invertible, the spectrum of $x$ contains 0. Choose a non-negative continuous function $g$ supported in $[-\sm, \sm]$ such that $g(0)=1$. Then $g(x)$ is non-zero. Since $A$ has Property (SP), there exists a non-zero projection $p$ in $\Her(g(x))$. Also that $A$ is simple, by Lemma 3.5.6 or \cite{Lin-C}, there exist non-zero projections $p_1\leq p$ and $q_1 \leq 1-p$, such that $p_1\sim q_1$. A corner of real rank 0 C*-algebra is again a real rank 0 C*-algebra, hence by Lemme 2.3 of \cite{EN-TA}, $(1-p)A(1-p)$ belongs to $\TA\mathscr{I}$. By the definition of $\TA\mathscr{I}$, there exist a projection $q\in (1-p)A(1-p)$ and C*-subalgebra $C\subset A$ of real rank 0, such that $1_C=q$ and:
\begin{enumerate}[(1)]
\item
$\|qf(x)q-y\|<\ep/4$, for some self-adjoint element $y\in C$.
\item
$1-p-q\pa q_1$.
\end{enumerate}
Since $q_1\sim p_1\leq p$, there exist some projection $r\leq p$ and a partial isometry $v$ such that $vv^*=1-p-q$ and $v^*v=r$. Now the identity of $A$ can be decomposed into the sum of orthogonal projections: $1=(p-r)+r+(1-p-q)+q$. We can write $f(x)$ into a matrix form according to this decomposition. Note that $f(x)=(1-p)f(x)(1-p)$, we have:
\[
f(x)=
\begin{pmatrix}
0&0&0&0\\
0&0&0&0\\
0&0& (1-p-q)f(x)(1-p-q) & (1-p-q)f(x)q\\
0&0&qf(x)(1-p-q) & qf(x)q \\
\end{pmatrix}
\]
Since $C$ has real rank 0 and $\|qf(x)q-y\|<\ep/4$ for some self-adjoint element $y\in C$, we could find a invertible self-adjoint element $b\in C$ such that $\|qf(x)q-b\|<\ep/2$. Let $a=(1-p-q)f(x)(1-p-q)$, $c=(1-p-q)f(x)q$. Let $Z$ be the matrix:
\[
\begin{pmatrix}
p-r&0&0&0\\
0&r&0&0\\
0&0& 1-p-q & -cb^{-1}\\
0&0& 0 & q \\
\end{pmatrix}
\]
Then by the same computation as in Lemma 3.1.5 of \cite{Lin-C}, we can show that:
\[
f(x)=
\begin{pmatrix}
0&0&0&0\\
0&0&0&0\\
0&0& a & c\\
0&0&c^* & b \\
\end{pmatrix}
\,=\,Z
\begin{pmatrix}
0&0&0&0\\
0&0&0&0\\
0&0& a-cb^{-1}c^* & 0\\
0&0& 0  & b \\
\end{pmatrix}
Z^*
\]
Now if we consider the element:
\[
x'\,=\,Z
\begin{pmatrix}
(\ep/2)(p-r)&0&0&0\\
0&0&(\ep/2) v*&0\\
0&(\ep/2) v& a-cb^{-1}c^* & 0\\
0&0& 0  & b \\
\end{pmatrix}
Z^*,
\]
We can check that $x'$ is an invertible self-adjoint element such that $\|f(x)-x'\|< \ep/2$. Hence $\|x-x'\|\leq \|x-f(x)\|+\|f(x)-x'\|<\ep$. Therefore $A$ has real rank 0.
\end{proof}

Hence we have the following corollary:
\begin{cor}
Let $A$ be an unital C*-algebra. Let $\af\colon G\rightarrow \Aut(A)$ be a finite group action with the weak tracial Rokhlin property, such that the crossed product $C^*(G,A,\af)$ is simple. Suppose $A$ has real rank 0, then $C^*(G,A,\af)$ also has real rank 0.
\end{cor}
\begin{proof}
$A$ has real rank 0 implies that $A$ has Property (SP). We need only to consider the case that $A$ is infinite dimensional, because any finite dimensional C*-algebra has real rank 0. Therefore the above statement follows from Theorem \ref{wTRP-Main} and Lemma \ref{RR0}
\end{proof}

\section*{Acknowledgement}
I would like to express my deep gratitude to Prof Phillips and Prof Weaver, my research supervisors, for their patient guidance, valuable suggestions and enthusiastic encouragement of this research work. I also would like to thank Prof McCarthy for encouraging me to publish this paper.


\begin{thebibliography}{14}
\bibitem{Ph-S}
N.Christopher Phillips, {\emph{Ottawa Summer School Course on Crossed Product C*-algebras}}, Lecture notes.
\bibitem{Ph-C}
N.Christopher Phillips, {\emph{Finite Cyclic Group Actions With The Tracial Rokhlin Property}}, arXiv:math/06090785v1
\bibitem{EPW}
Siegfried Echterhoff, Wolfgang Lueck, N. Christopher Phillips, Samuel Walters, {\emph{The structure of crossed products of irrational rotation algebras by finite subgroups of $SL_2 (Z)$}}, Journal für die reine und angewandte Mathematik, arXiv:math/0609784
\bibitem{Ph-F}
N.Christopher Phillips, {\emph{The tracial Rokhlin property for actions of finite groups on C*-algebras}}, American Journal of Mathematics - Volume 133, Number 3, June 2011, pp. 581-636
\bibitem{Iz1}
 Masaki Izumi, {\emph{Finite group action on C*-algebras with Rokhlin property I}}, Duke Math, J. 122(2004), no.2, 233-180
\bibitem{RS1}
M. R$\o$rdam, E. St$\o$mer, {\emph{Classification of Nuclear C*-Algebras. Entropy in Operator Algebras}},  Encyclopedia of Mathematical Sciences, Springer, ISBN-10: 3540423052
\bibitem{Cu1}
J. Cuntz, {\emph{The structure of multiplication and addition in simple C*-algebras}}, Math. Scand. 40 (1977), 215223. 9
\bibitem{Os1}
Hiroyuki Osaka, {\emph{SP-Property For a Pair of C*-algebras}}, J.Operator theory 46(2001), 159-171
\bibitem{Lin-TR1}
Huaxin Lin, {\emph{The tracial topological rank of C*-algebras,}} Proc. London Math. Soc. (3) \textbf{83} (2001), no. 1, 199--234.
\bibitem{Lin-TR2}
Shanwen Hu, Huaxin Lin and Yifeng Xue, {\emph{The tracial topological rank of C*-algebras II}}, Indiana Univ. Math. J. \textbf{53} (2004), no. 6, 1578--1603.
\bibitem{Lin-C}
Huaxin Lin, {\emph{ An introduction to the classification of amenable C*-algebras,}} World Scientific Publishing Co., Inc., River Edge, NJ, 2001. xii+320 pp. ISBN: 981-02-4680-3
\bibitem{ORT}
Eduard Ortega, Mikael Rordam, Hannes Thiel, {\emph{The Cuntz semigroup and comparison of open projections},} Journal of Functional Analysis Volume 260, Issue 12, 15 June 2011, Pages 3474–3493 arXiv:1008.3497 [math.OA]
\bibitem{EN-TA} 
G.A.Elliott, Z Niu, {\emph{On tracial approximation}, } Journal of Functional Analysis Volume 254, Issue 2, 15 Jan 2008, Pages 396-440.
\bibitem{OPed}
Dorte Olesen, Gert K Pedersen, {\emph{Applications of the Connes spectrum to C*-dynamical systems III,}} Journal of Functional Analysis Volume 45, Issue 3, 15 February 1982, Pages 357–390.
\bibitem{S-IS}
Adam Sierakowski, {\emph{The ideal structure of reduced crossed products}}, M$\ddot{u}$nster Journal of Mathematics  3 (2010), 237–262
\bibitem{R-F}
Rieffel, Marc A, {\emph{Actions of finite groups on C*-algebras}}. Math Scand. 47 (1980), 157-176
\end{thebibliography}
\end{document}